\documentclass[12pt]{article}

\usepackage{amssymb}
\usepackage{amsmath,amsthm}
\usepackage[english]{babel}
\usepackage[T1]{fontenc}
\usepackage[a4paper,top=2.5cm,bottom=3cm,left=2.5cm,right=2.5cm,marginparwidth=1.75cm]{geometry}

\usepackage{microtype}

\usepackage{enumitem} %

\usepackage{mathtools}
\usepackage{graphicx,adjustbox}
\usepackage[dvipsnames]{xcolor}
\usepackage{bbm}
\usepackage[colorlinks=true, allcolors=black]{hyperref}
\usepackage[normalem]{ulem}
\usepackage{cancel}
\usepackage{varwidth}  %
\usepackage{booktabs}

\usepackage{xparse}
\NewDocumentCommand{\grad}{e{_^}}{%
  \mathop{}\!%
  \nabla
  \IfValueT{#1}{_{\!#1}}%
  \IfValueT{#2}{^{#2}}%
}

\newcommand{\R}{\bbR}

\DeclareMathOperator{\tr}{tr}
\DeclareMathOperator*{\argmin}{arg\, min}

\DeclareMathOperator{\sign}{sign}

\usepackage{xspace}

\usepackage{dsfont}

\newtheorem{theorem}{Theorem}[section]
\newtheorem{corollary}[theorem]{Corollary}
\newtheorem{lemma}[theorem]{Lemma}
\newtheorem{proposition}[theorem]{Proposition}
\newtheorem{assumption}[theorem]{Assumption}

\theoremstyle{definition}
\newtheorem{definition}[theorem]{Definition}
\usepackage{textcomp}
\newtheorem{remarkx}[theorem]{Remark}
\newtheorem{examplex}[theorem]{Example}
\newenvironment{remark}
{\pushQED{\qed}\remarkx} %
{\popQED\endremarkx}
\newenvironment{example}
{\pushQED{\qed}\examplex} %
{\popQED\endexamplex}

\DeclarePairedDelimiter{\norm}{\lVert}{\rVert}
\DeclarePairedDelimiter{\bra}{(}{)}

\DeclarePairedDelimiter{\set}{\{}{\}}
\DeclarePairedDelimiter{\skp}{\langle}{\rangle}
\DeclareMathAlphabet{\mathup}{OT1}{\familydefault}{m}{n}

\usepackage{ifthen}
\newlength{\leftstackrelawd}
\newlength{\leftstackrelbwd}
\def\leftstackrel#1#2{\settowidth{\leftstackrelawd}%
{${{}^{#1}}$}\settowidth{\leftstackrelbwd}{$#2$}%
\addtolength{\leftstackrelawd}{-\leftstackrelbwd}%
\leavevmode\ifthenelse{\lengthtest{\leftstackrelawd>0pt}}%
{\kern-.5\leftstackrelawd}{}\mathrel{\mathop{#2}\limits^{#1}}}

 \def\calE{{\mathcal E}} \def\calF{{\mathcal F}}
\def\calG{{\mathcal G}} \def\calH{{\mathcal H}} \def\calI{{\mathcal I}}
  \def\calL{{\mathcal L}}
\def\calM{{\mathcal M}}  
\def\calP{{\mathcal P}}  
  
 \def\calW{{\mathcal W}} \def\calX{{\mathcal X}}
\def\calY{{\mathcal Y}} 

 \def\fw{{\mathrm W}}

 \def\bbE{{\mathbb E}} 
  
 \def\bbK{{\mathbb K}} 
 \def\bbN{{\mathbb N}} 
\def\bbP{{\mathbb P}}  \def\bbR{{\mathbb R}}
\def\bbS{{\mathbb S}} \def\bbT{{\mathbb T}} 
  
 \def\bbZ{{\mathbb Z}} 

\usepackage{graphicx} %
\usepackage[pdftex]{pict2e}

\renewcommand{\#}{\sharp}

\newcommand{\dist}{\mathrm{dist}}
\newcommand{\proj}{\mathrm{proj}}

\newcommand{\divr}{\mathrm{div}}

\makeatletter
\let\@fnsymbol\@arabic
\makeatother

\begin{document}

\title{Solutions of stationary McKean-Vlasov equation on a high-dimensional sphere and other Riemannian manifolds}

\author{Anna Shalova\thanks{\href{mailto:a.shalova@tue.nl}{a.shalova@tue.nl}} \quad Andr\'e Schlichting\thanks{\href{mailto:andre.schlichting@uni-ulm.de}{andre.schlichting@uni-ulm.de}%
}}
\date{\normalsize ${}^1$Department of Mathematics and Computer Science,\\
  Eindhoven University of Technology \\
${}^2$Institute of Applied Analysis, Ulm University}

\maketitle

\def\ourkeywords{McKean-Vlasov equation, bifurcations, phase transition, nonlocal PDEs, interacting particle systems, PDEs on manifolds.}

\begin{abstract}
    We study stationary solutions of McKean-Vlasov equation on a high-dimensional sphere and other compact Riemannian manifolds. We extend the equivalence of the energetic problem formulation to the manifold setting and characterize critical points of the corresponding free energy functional. On a sphere, we employ the properties of spherical convolution to study the bifurcation branches around the uniform state. We also give a sufficient condition for an existence of a discontinuous transition point in terms of the interaction kernel and compare it to the Euclidean setting. We illustrate our results on a range of system, including the particle system arising from the transformer models and the Onsager model of liquid crystals.
\end{abstract}

\tableofcontents

\section{Introduction}
McKean-Vlasov equation arises as a mean-field limit of various stochastic interacting particles systems. Such systems describe phenomena of different nature and have applications in fields varying from liquid crystals \cite{carrillo2020long, Vollmer2017} and statistical mechanics \cite{MartzelAslangul2001} to opinion dynamics \cite{HegselmannKrause2002}, mathematical biology \cite{KellerSegel1971, BurgerCapassoMorale2007}, galactic dynamics~\cite{binney2008}, droplet growth~\cite{ConlonSchlichting2019}, plasma physics~\cite{bittencourt1986fund}, and synchronisation~\cite{kuramoto1981rhythms}. In addition, recently, interacting particles systems found a whole set of applications in theoretical machine learning \cite{sirignano2020mean, rotskoff2022trainability, geshkovski2024mathematical}. Several of the above-mentioned applications are set on Riemannian manifolds, dominantly on a high-dimensional sphere~\cite{Vollmer2017, geshkovski2024mathematical}. Even though the solutions of the McKean-Vlasov equation are relatively well-studied in~$\bbR^n$ or the flat torus, the scope of work concerning McKean-Vlasov equation in a manifold setting is very limited. 

In this paper we characterize the set of measure-valued solutions $\rho \in \calP_{ac}(\calM)$ of the stationary McKean-Vlasov equation:
\begin{equation}
    \label{eq:mckean-vlasov}
    \gamma^{-1}\Delta\rho + \divr(\rho \nabla_x W(x, \cdot) *\rho) =0,
\end{equation}
 on a compact Riemannian manifold $\calM$ in general and on sphere $\calM =\bbS^{n-1}$ of arbitrary dimension bin particular. Solutions of this equation correspond to the densities which balance the first, \emph{diffusion} term and the second, \emph{interaction} term. 
 The function $W: \calM \times \calM \to \bbR$ is called an \emph{interaction kernel} and is assumed to be symmetric $W(x,y) = W(y,x)$ throughout this paper. 
 Depending on the direction of $\nabla W$, the interaction term can model both \emph{attractive} or \emph{repulsive} forces. 

The parameter $\gamma \in \bbR_+$, called \emph{inverse temperature}, expresses how much priority is given to the diffusion term. 
Formally, for $\gamma \to 0$ the impact of the interaction term becomes negligible; and as a result, we expect that the set of solutions of \eqref{eq:mckean-vlasov} will coincide with the kernel of the Laplace-Beltrami on $\calM$, which are constant with respect to the volume measure.
Similarly, for $\gamma \to \infty$ the priority is given to the interaction term and the structure of the set of the solutions can vary depending on the properties of the interaction kernel $W$. We study the case of small $\gamma$ for a general compact Riemannian manifold. In case of $\calM=\bbS^{n-1}$ the knowledge of a suitable basis of $L_2(\bbS^{n-1})$ and its behavior under convolution operations allows us to characterize the behaviour of certain solutions for a larger range of $\gamma \in \bbR_+$.

We begin our analysis by establishing equivalence between solutions of the stationary McKean-Vlasov equation \eqref{eq:mckean-vlasov} and critical points of the free energy functional $\calF_\gamma: \calP(\calM) \to \bbR$ (see Proposition~\ref{prop:equivalence}) which for any admissible $\calM$ consists of
\begin{equation}
\label{eq:free-energy}
    \calF_\gamma(\mu) := \gamma^{-1}\calE(\mu) + \calI(\mu) \,.
\end{equation}
where $\calE$ is the relative entropy with respect to the normalized volume measure $m$:
\begin{equation}
\label{eq:entropy}
\calE(\mu) := \begin{cases}
    \int_{\calM} \rho \log \rho \,d{m} & \text{ if } \mu \text{ admits a positive density } \rho \text{ w.r.t. } m,  \\
    +\infty &\text{otherwise.}
\end{cases}
\end{equation}
The second term $\calI: \calP(\calM) \to \bbR$ is called the interaction energy and denoted by
\begin{equation}
\label{eq:interaction-energy}
\calI(\mu) := \frac12\int_{\calM\times \calM} W(x, y )d\mu(x)d\mu(y).
\end{equation}
Using this equivalence we prove existence of solutions for arbitrary $\gamma\in\bbR_+$ and give a sufficient condition for the uniqueness of the solution for small $\gamma$.

Additional symmetry assumptions on the space $\calM$ and the interaction kernel $W$ can help to give a more explicit characterization of the solutions of \eqref{eq:mckean-vlasov} like it was done in case of a torus in \cite{carrillo2020long}. In \cite{carrillo2020long}, the authors showed that for an interaction kernel of form $W(x, y) = W(x-y)$ on a torus $\bbT^{n}$ the Fourier decomposition of the interaction kernel $W$ can be used to establish existence of bifurcation branches as well as characterize the phase transition of \eqref{eq:mckean-vlasov}. In this work we employ similar techniques to study the solutions of the stationary McKean-Vlasov equation on a sphere of arbitrary dimension $\calM=\bbS^{n-1}$. We study the bifurcation branches around the uniform state $\bar\rho$ and give a sufficient condition for the existence of a discontinuous transition point in terms of the spherical harmonics decomposition of the interaction kernel in case of a radially-symmetric kernel $W(x, y) = W(\left<x, y\right>)$. 

To characterize non-trivial stationary measures of the McKean-Vlasov equation we use another equivalent formulation (see Proposition~\ref{prop:equivalence}), namely the characterization of the invariant measures to~\eqref{eq:mckean-vlasov} in terms of the zeroes of the Gibbs-map $F: \bbR_+ \times L^2(\calM) \to L^2(\calM)$:
\begin{equation}
\label{eq:gibbs-map}
    F(\gamma, \rho) = \rho - \frac{1}{Z(\gamma, \rho)}e^{-\gamma W*\rho} \,,
\end{equation}
where $Z(\gamma, \rho)$ is a normalization constant $Z(\gamma, \rho) = \int_{\calM}e^{-\gamma W*\rho}dm$. Applying results from the bifurcation theory to the Gibbs map, we show that the bifurcation points can be expressed in terms of the spherical harmonics decomposition of $W$ and the corresponding invariant measures can be characterized in terms of the corresponding spherical basis functions. The same decomposition in combination with the known structure of the spherical harmonics allows us to study the behaviour of minimizers around the phase transition point. 

We apply our findings to a number of models of different nature. We begin by studying so-called noisy transformer model, which can be interpreted as stochastically perturbed continuous-time self-attention model \cite{geshkovski2024mathematical}. Self-attention is a key building block of transformers, the state-of-the-art large language models. We characterize invariant measures of the noisy transformers as well as calculate the critical noise ratio above which no prior information is preserved. We also study the Onsager model for liquid crystals, which also arises in mathematical biology, and generalize findings of \cite{WachsmuthThesis06,Vollmer2017} to the case of the unit sphere of an arbitrary dimension. Finally, we study the noisy Hegselmann–Krause model for opinion dynamics adapted to the spherical domain. 

All of the models can formally be interpreted as mean-filed limits of the corresponding particles system~\cite{McKean1966,Oelschlaeger1984,oelschlager1989derivation}. 
The corresponding evolution equation for the law has the structure:
\[
\partial_t\rho = \nabla \cdot\left(\rho \nabla \frac{\delta \calF_\gamma}{\delta\rho}\right),
\]
where $\frac{\delta \calF_\gamma}{\delta\rho}$ is the Fréchet derivative of the free energy functional from~\eqref{eq:free-energy}. PDEs of this form posed on the space of probability measures with bounded second moments belong to a larger class of systems, namely gradient flows. We refer the reader to \cite{ambrosio2005gradient, santambrogio2015optimal} for the general theory of gradient flows on the state space $\R^d$. On manifolds the general theory is not fully developed, but it is expected to carry over. 
For instance on manifolds of positive curvature \cite{erbar2010heat} establishes the gradient flow formulation of the heat equation driven by relative entropy, albeit without interaction term. 
Due to the regular structure of the sphere, we argue that the same approaches might be applicable to rigorously prove the limiting behavior of the interacting particles systems posed on a sphere. In this paper we treat the stationary version of the McKean-Vlasov equation but the convexity properties established in Section~\ref{sec:convexity}, generalizing results from~\cite{sturm2005convex}, may also be of use for the characterization of the gradient-flow solutions of the non-stationary equation.

\subsection{Main results}
In this section we give an overview our main contributions. Our results are two-fold: we first study the solutions of the stationary McKean-Vlasov equation \eqref{eq:mckean-vlasov} on a compact connected Riemannian manifold without boundary, and in the second part we employ the symmetry properties of the unit sphere endowed with the natural topology to give a more explicit characterization of the solutions in terms of the spherical harmonics basis.
\paragraph{Compact Riemannian manifold.} Let $\calM$ be a compact connected Riemannian manifold without boundary and let the interaction kernel $W: \calM\times\calM \to \bbR$ be continuous, then the following result holds (see Theorem~\ref{th:convexity-M} and Corollary~\ref{cor:convergence-min}).
\begin{theorem}[Existence and uniqueness of solutions]
    For any $\gamma \in \bbR_+$ there exist a solution $\rho_\gamma$ of \eqref{eq:mckean-vlasov} and $\rho_\gamma \in H^1(\calM) \cap \calP_{ac}(\calM)$. In addition, if the curvature of the manifold is bounded from below $\operatorname{Ric}(\calM) \geq \lambda$, $W$ is twice-differentiable and there exist $\alpha > -\gamma^{-1}\lambda$ such that $W$ satisfies
    \[
\partial^2_t W\left(\exp_x vt, \exp_y ut\right) \geq \alpha (\|v\|^2 + \|u\|^2)
    \]
    for all $x, y \in \calM, \ v\in T_x\calM, u \in T_y\calM$, then $\rho_\gamma$ is a unique solution of \eqref{eq:mckean-vlasov}.
\end{theorem}
In fact we don't require $W$ to be everywhere twice-differentiable but only need the bound on the lower-second derivative. The proof relies on the geodesic convexity condition of the free energy functional \eqref{eq:free-energy}. 
\paragraph{Sphere $\bbS^{n-1}$.}
In case of the high-dimensional sphere we impose more assumptions on the interaction kernel, namely we ask $W$ to be rotationally symmetric, namely by abuse of notation to take the form $W(x,y) = W(\left<x, y\right>)$ with $W:[-1,1]\to \R$. In this case, due to the symmetric structure of the unit sphere and the interaction kernel one can show that the uniform state $\bar\rho$ is always a solution of \eqref{eq:mckean-vlasov}. Employing the properties of the spherical convolution we are able to characterize non-trivial branches of solutions in terms of the spherical harmonics decomposition of the kernel. Components of the spherical harmonics decomposition are projections of the function on the symmetric spherical harmonics basis functions $Y_{k,0}$. An explicit form is given in the Definition~\ref{def:spherical-decomposition}. 

\begin{definition}[Spherical harmonics decomposition, see Definition \ref{def:spherical-decomposition}]
\label{def:sph-decomposition-intro}
    Let $W:\bbS^{n-1}\times \bbS^{n-1} \to \bbR$ be a rotationally symmetric kernel, then the spherical harmonics decomposition of $W$ is defined as
    \[
    \hat{W}_k = \alpha_k \int_{\bbS^{n-1}}W(\skp{x_0,\cdot}) Y_{k, 0} \,d\sigma,
    \]
    where $\sigma$ is the uniform measure on a sphere, $x_0\in \bbS^{n-1}$ an arbitrary reference point, $Y_{k, 0}$ are the spherical harmonics and $\alpha_k$ is the normalization constant for $k\in \bbN$. 
\end{definition}
We show that if the spherical decomposition is non-positive, under certain structural assumptions, which we discuss in Section \ref{ssec:InteractionSphere}, there exist bifurcation curves around the uniform state.  Our result can be summarized in the following formal theorem (for more details see Theorem \ref{th:bifurcations}).
\begin{theorem}[Bifurcations]
\label{th:bifurcations-intro}
Let $W \in C_b \cap H^1$ be a rotationally symmetric interaction kernel. 
If there exists $k\in \bbN$ with unique negative value $\hat W_k < 0$, that is $\forall j\in \bbN\setminus\set{k}: W_j\ne W_k$, then there exists a non-trivial branch of solutions $\rho_\gamma \in L_2(\bbS^{n-1})$ of the form
\[
\rho_\gamma(t) = \bar\rho + f(t)Y_{k, 0} + o(f(t)), \qquad \gamma(t) = \gamma_k + \mu(t),
\]
on some neighborhood $t \in (-\delta, \delta)$ around the bifurcation point $\gamma_k = -\frac{1}{\hat W_k}$,
where $\bar\rho$ is the uniform state, $Y_{k, 0}$ is the corresponding spherical harmonic and $f, \mu$ are continuous functions on $(-\delta, \delta)$ satisfying $f(0) = 0, \ \mu(0) =0$. 
\end{theorem}
Bifurcation theory describes continuous curves of solutions branching from the uniform state. These solutions however are not guaranteed to be (global) minimizers of the free energy functional \eqref{eq:free-energy}. Indeed, it may be the case that above certain value $\gamma > \gamma_c$ the uniform measure is no longer a global minimizer of \eqref{eq:free-energy} and a different configuration is preferable from the energy-minimization perspective. This phenomena is called phase transition and the value $\gamma_c$ where the uniform state stops being unique minimizer of the free energy is called a phase transition point (see Definition~\ref{def:transition-point}. 
We characterize the phase transition of the stationary McKean-Vlasov equation \eqref{eq:mckean-vlasov} for a certain class of the interaction kernels. We give a simplified version of the sufficient condition for a discontinuous phase transition here. See the detailed description in the Assumption \ref{assum:pt-general} and Theorem \ref{th:pt}.
\begin{assumption}[Competitor in spherical harmonics]
   \label{assum:resonance-intro}
   Let $W$ be a rotationally symmetric interaction kernel and let $k\in \bbN$ be such that $\hat W_k= \min_l \hat W_l$ is among the smallest component of the spherical harmonics decomposition of $W$. Let $N_{\hat W_k}$ be the set of the indexes of all components with $\hat W_n = \hat W_k:$ 
   \[
   N_{W_k}= \{n\in \bbN: \hat W_n = \hat W_k\},
   \] 
   The interaction potential $W$ satisfies the resonance condition if there exists a linear combination $v = \sum_{l\in N_{W_k}} \alpha_l Y_{l,0}$ satisfying:
   $ \int \hat v^3 \,d\sigma \neq 0.
   $
\end{assumption}
In particular we show that the above assumption is satisfied, for example, whenever the minimum is achieved for $k = 2$ or $k=4$, which is the case in the Examples of Sections~\ref{ssec:Onsager},~\ref{ssec:opinion} and~\ref{ssec:localized}.
In this sense, single modes can resonate with themselves. 
Under the above assumption we are able to prove existence of the discontinuous transition point. 
    \begin{theorem}[Phase transitions]
        Let the interaction kernel satisfy the resonance Assumption~\ref{assum:resonance-intro}, then there exists a discontinuous phase transition point $0<\gamma_c < -\frac{1}{\min_{n\in\bbN} \hat W_n}$.
    \end{theorem}
Note that in this case $\gamma_c$ is strictly smaller then any of the bifurcation points characterized in Theorem \ref{th:bifurcations-intro}, implying that in the bifurcation points the uniform measure is not a global minimizer of the free energy functional \eqref{eq:free-energy}. 
\subsection{Literature Review}

\paragraph{McKean-Vlasov equation as a mean-field limit.}

Mean-field limits of particles system is a vast area of research, we refer to several recent results in this direction. A number of works treat interaction and diffusion systems separately. Namely, the mean-field convergence of Vlasov system (without interaction) under various assumptions is reviewed in \cite{jabin2014review}. 
Convergence of the system of interacting particles (with noise) goes back to~\cite{McKean1966} with rigorous derivations with more and more singular interaction kernels in~\cite{Oelschlaeger1984,oelschlager1989derivation,Stevens2000} and quantitative limits in~\cite{duerinckx2016mean, Serfaty2020mean} for Riesz and Coulomb-type (repulsive) interactions, also see the overview \cite{golse2016dynamics} and the recent work~\cite{bresch2023mean} for a mean-field with singular kernels.
Recent innovations consider the question of uniform in time propagation of chaos in mean field limit of interacting diffusions with smooth kernels as for instance in~\cite{monmarche2017long} and references therein and upto the bifurcation point	 in~\cite{DelgadinoGvalaniPavliotisSmith2023}, optimal quantitative results as first established in~\cite{Lacker2023}, or revisit connection to large deviation principles~\cite{DawsonGaertner1989,hoeksema2024large}.

\paragraph{PDEs and free energies on manifolds.}

Well-posedness of the pure interaction systems on Riemannian manifolds have been studied in \cite{fetecau2021well, wu2015nonlocal}. Under the bounded curvature assumption the long-term behaviour of the same system have been established in \cite{fetecau2023long}. Relaxation of the manifold-restricted aggregation model has been introduced and studied in \cite{patacchini2021nonlocal}. On a sphere, well-posedness of the aggregation model is established in \cite{fetecau2021intrinsic}. In \cite{fetecau2023equilibria} the authors study the aggregation PDE on Cartan-Hadamar (hyperbolic) manifolds.

For the manifolds with negative curvature the it is also possible to establish well-posedness of the aggregation model in the presence of diffusion term. Stationary solutions of McKean-Vlasov equation on hyperbolic manifolds are characterized in \cite{fetecau2023equilibria, fetecau2023ground, carrillo2024existence}.

A few relevant results concern the free energies corresponding to the evolution equations on manifolds. The geodesic convexity of the entropic term and potential energy is established in  \cite{otto2005eulerian, sturm2005convex}. We give a more detailed description of~\cite{sturm2005convex} in Section~\ref{sec:convexity}. In  \cite{erbar2010heat}, the author shows existence and uniqueness of gradient flow solutions of the heat equations on manifolds of positive curvature. The general formalism of gradient flows for internal energies on the space of measures over a Riemannian manifold is discussed in~\cite{Villani2008}.

\paragraph{Bifurcations and phase transitions.}
Bifurcation theory dates back to the results formulated in \cite{CrandallRabinowitz1971}, for a general theoretical overview we refer the reader to the book of Kielhoefer \cite{Kielhoefer2012}. 
On a torus bifurcations of the free energy functional \eqref{eq:free-energy} have been studied in \cite{carrillo2020long} and in the presence of two local minima the existence of saddle point was proven~\cite{GvalaniSchlichting2020}. See also~\cite{CarrilloGvalani2021} for a generalization to nonlinear diffusion-aggregation equations.
On $\bbS^2$ bifurcations of the Onsager energy are characterized in~\cite{fatkullin2005critical, WachsmuthThesis06, lucia2010exact, Vollmer2017}. 

Phenomenon of phase transition has been show to appear in systems of different nature, see for example
\cite{PoschNarenhoferThirring1990,BarbaroCanizoCarrilloDegond2016, DegondFrouvelleLiu2015,Tugaut2014, Vollmer2017}. Phase transition of the McKean-Vlasov equation on a torus has been studied in
\cite{ChayesPanferov2010}, the authors introduce concepts of continuous and discontinuous transition points and study their properties in terms of the interaction kernel. Explicit conditions of continuous and discontinuous phase transition in terms of the Fourier decomposition of the kernel are introduced in 
\cite{carrillo2020long}. Phase transition of McKean-Vlasov equation of weakly coupled Hodgkin-Huxley oscillators is characterized in \cite{vukadinovic2023phase}. In \cite{delgadino2021diffusive}, the authors discuss the mean-field behaviour of systems exhibiting phase transition. 

\subsection*{Acknowledgments}
The authors are grateful to Hugo Melchers for the help concerning calculations in Section~\ref{sec:examples}. The authors are also thankful to Rishabh Gvalani, Jasper Hoeksema, Greg Pavliotis, Mark Peletier and Jim Portegies for helpful discussions. Andr\'e Schlichting is supported by the Deutsche Forschungsgemeinschaft (DFG, German Research Foundation) under Germany's Excellence Strategy EXC 2044-390685587, Mathematics M\"unster: Dynamics--Geometry--Structure. Anna Shalova is supported by the Dutch Research Council (NWO), in the framework of the program ‘Unraveling Neural Networks with Structure-Preserving Computing’ (file number OCENW.GROOT.2019.044).

\section{Compact Riemannian manifold}
\label{sec:general}

Throughout this section we assume that $\calM$ is a compact connected Riemannian manifold without boundary. 
We study the weak solutions on $\calM$ of the stationary McKean-Vlasov equation~\eqref{eq:mckean-vlasov}, that is
\begin{equation*}
	\gamma^{-1}\Delta\rho + \divr(\rho \nabla_x W(x, \cdot) *\rho) =0 \,,
\end{equation*}
where the operators $\nabla, \ \divr \text{ and } \Delta$ are manifold gradient, divergence and Laplace-Beltrami operator respectively and are rigorously defined in Appendix~\ref{sec:geometry} and $*$ denotes the measure convolution
\[
(W*\rho)(x) = \int_{\calM} W(x, y)\rho(y)dm.
\]
For a Riemannian manifold with metric $g$, given the interaction kernel $W\in H^1(\calM\times\calM)$ (see Appendix~\ref{ssec:SobolevMfds} for the notion of Sobolev spaces) the weak solutions are defined in the following sense.
\begin{definition}[Weak solution]\label{def:weak:mv}
	A function $\rho\in H^1(\calM) \cap \calP_{ac}(\calM)$ is a weak solution of \eqref{eq:mckean-vlasov} if for every $\phi \in H^1(\calM)$ it satisfies
	\[
	\gamma^{-1}\int_{\calM}g(\nabla \rho, \nabla \phi)d\sigma + \int_{\calM} g(\rho \nabla\phi, \nabla_x W(x,\cdot) *\rho) d\sigma  =0.
	\] 
\end{definition}
The structure of this section is the following: we first establish three equivalence formulations for weak solution in the sense of Definition~\ref{def:weak:mv} in Section~\ref{sec:formulations}. We then proceed by proving existence of minimizers of the free energy functional $\calF$ in Section~\ref{sec:existence}. Finally, we introduce a convexity criterion for $\calF$ and derive a sufficient condition for the uniqueness of the minimizers in Section~\ref{sec:convexity}.

\subsection{Equivalent characterizations of stationary states}
\label{sec:formulations}
In this section we reformulate the problem of solving the stationary McKean-Vlasov equation as a fixed-point problem of the Gibbs map $F$ as defined in \eqref{eq:gibbs-map} and as a minimization problem of the free energy functional defined in \eqref{eq:free-energy}. First we note that due to the smoothing effect of the convolution all the zeros of the Gibbs map are smooth, namely the following Lemma holds.
\begin{lemma}
	\label{lemma:gibbs-H1}
	Let $\gamma \in \bbR_+$ and let $W \in C_b(\calM \times\calM) \cap H^1(\calM \times\calM)$, then any $\rho \in L^1(\calM)$ satisfying $F(\rho, \gamma) = 0$ is an $H^1(\calM)$ function.%
\end{lemma} 
\begin{proof}
	We begin by showing $\rho \in L^2(\calM)$. From the boundedness of the kernel we obtain the  following estimate
	\[
	\|W * \rho \|_\infty =  \left\|\int   W(x, y)\rho(y)dm(y)\right\|_\infty%
	\leq \|W\|_{L_\infty(\calM\times\calM)} \|\rho\|_{L_1(\calM)}.
	\]
	Any zero of the Gibbs map satisfies almost everywhere 
	\[
	\rho(x) = \frac{1}{Z(\gamma, \rho)} e^{-\gamma (W *\rho)(x)},
	\]
	implying that
	\begin{equation}
		\label{eq:rho-infty}
		\|\rho\|_\infty = \left\|\frac{1}{Z(\gamma, \rho)} e^{-\gamma W *\rho}\right\|_\infty = \frac{1}{Z(\gamma, \rho)}\left\| e^{-\gamma W *\rho}\right\|_\infty \leq \frac{1}{Z(\gamma, \rho)}e^{\gamma \|W \|_\infty}  = m(\calM)^{-1}e^{2\gamma \|W \|_\infty},
	\end{equation}
	where we used that $Z(\gamma, \rho)\geq \int e^{-\gamma \|W \|_\infty}dm = m(\calM)e^{-\gamma \|W \|_\infty} > 0$. As a result we conclude that $\rho$ is square integrable $\|\rho\|_2 \leq m(\calM)\|\rho\|^2_\infty < \infty$.
	
	Now, we show that $\nabla \rho \in L_2(T\calM)$. First of all note that the gradient exists and satisfies
	\begin{align*}
		\nabla \rho(x) &= \frac{1}{Z(\gamma, \rho)} \nabla e^{-\gamma (W *\rho)(x)} = - \frac{\gamma e^{-\gamma (W *\rho)(x)}}{Z(\gamma, \rho)} \int_\calM \nabla_x W(x, y)  \rho(y)dm(y)\\
		&=
		- \frac{\gamma e^{-\gamma (W *\rho)(x)}}{Z(\gamma, \rho)} (\nabla_x W\ast \rho)(x) \,.
	\end{align*}
	As a result we get the following bound
	\begin{align}
		\MoveEqLeft   
		\int_{\calM}g(\nabla \rho, \nabla \rho)dm \leq \frac{\gamma^2e^{2\gamma\|W*\rho\|_{\infty}}}{Z(\gamma, \rho)^2} \int_{\calM}g_x\bra*{(\nabla_x W\ast \rho)(x), (\nabla_x W\ast \rho)(x)} dm(x) \notag \\
		&\leq\frac{\gamma^2e^{2\gamma\|W*\rho\|_{\infty}}}{Z(\gamma, \rho)^2}\|\rho\|^2_{\infty}\int_{\calM^3}\mkern-4mu g_x\bigl( \nabla_x W(x, y),  \nabla_x W(x, z)\bigr) (dm)^3 \notag\\
		&\leq \frac{\gamma^2e^{2\gamma\|W*\rho\|_{\infty}}}{2Z(\gamma, \rho)^2}\|\rho\|^2_{\infty}
			\int_{\calM^3} \Bigl(g_x\bigl( \nabla_x W(x, y),  \nabla_x W(x, y) \bigr) \notag \\
			&\hspace{16em}+ g_x\bigl( \nabla_x W(x, z),  \nabla_x W(x, z) \bigr)\Bigr)(dm)^3 \notag\\
		&\leq \frac{\gamma^2e^{2\gamma\|W*\rho\|_{\infty}}}{2Z(\gamma, \rho)^2}\|\rho\|^2_{\infty} m(\calM) \int_{\calM^3}\Bigl(g_x\bigl( \nabla_x W(x, y),  \nabla_x W(x, y) \bigr)  \notag \\
        &\hspace{16em} + g_y\bigl( \nabla_y W(x, y),  \nabla_y W(x, y) \bigr)\Bigr)(dm)^3 \notag\\
		&\leq \frac{\gamma^2e^{2\gamma\|W*\rho\|_{\infty}}}{2Z(\gamma, \rho)^2}\|\rho\|^2_{\infty} m(\calM) \int_{\calM\times \calM} g^{\calM\times \calM} (\nabla W(x, y), \nabla W(x, y))(dm)^2 \notag \\
		&\leq\frac{\gamma^2e^{2\gamma\|W*\rho\|_{\infty}}}{2 Z(\gamma, \rho)^2}\|\rho\|^2_{\infty} m(\calM)\|W\|_{H^1} \,\label{eq:rho-h1}
	\end{align}
	where we use the product metric tensor $g^{\calM\times \calM}$ in the second last estimate (see Appendix~\ref{ssec:ProductMfds}).
\end{proof}
\begin{remark}
	In Euclidean setting the solutions of \eqref{eq:mckean-vlasov} are smooth functions $\rho \in C^\infty$, see for example \cite[Theorem 2.3]{carrillo2020long}. We argue that the same reasoning applies to the Riemannian manifold case and the solutions have in fact higher regularity. The main argument of the proof is the regularity of the 'convolution' which can be carried out in charts. Since it is not the main focus of the paper and is not required for the further analysis we do not provide the proof here.
\end{remark}
Estimates derived in the proof of Lemma \ref{lemma:gibbs-H1} also allow to characterize the limiting behavior of the minimizers for $\gamma \to 0$.
\begin{corollary}
	\label{cor:gibbs-gamma0}
	Let $W \in C_b(\calM \times\calM) \cap H^1(\calM \times\calM)$, and assume that for all $\gamma \in [0, M)$ there exists $\rho_\gamma \in H^1$ such that $(\gamma,\rho_\gamma)$ is a zero of the Gibbs map \eqref{eq:gibbs-map}, then 
	\[
	\lim_{\gamma\to 0}  \|\rho_\gamma - \bar \rho\|_{H^1} = 0 \,,
	\]
	where $\bar \rho = \frac{1}{m(\calM)}$ is the uniform state.
\end{corollary}
\begin{proof}
	Since $\bar\rho$ is a constant function, expanding $\|\rho_\gamma - \bar \rho\|_{H^1}$ we get
	\[
	\|\rho_\gamma - \bar \rho\|_{H^1} = \|\rho_\gamma - \bar \rho\|_{L_2} + \|\nabla\rho_\gamma \|_{L_2(T\calM)}. 
	\]
	Analogously to \eqref{eq:rho-infty}, we also have the lower bound on $\|\rho_\gamma\|_\infty$:
	\begin{equation*}
		\|\rho_\gamma\|_\infty \geq  \frac{1}{Z(\gamma, \rho)}e^{-\gamma \|W \|_\infty}  = m(\calM)^{-1}e^{-2\gamma \|W \|_\infty}.
	\end{equation*}
	and as a result the $L_2$ norm can be bounded as
	\[
	\|\rho_\gamma - \bar \rho\|^2_{L_2} \leq m(\calM)\|\rho_\gamma - \bar \rho\|^2_\infty \leq \bar\rho \left((1 - e^{-2\gamma \|W \|_\infty})^2 + (e^{2\gamma \|W \|_\infty}-1)^2\right) \leq 16\gamma^2\bar\rho^2\|W \|_\infty^2\,,
	\]
	which vanishes for $\gamma\to 0$. In addition, the bound \eqref{eq:rho-h1} combined with the upper bound on~$\|\rho_\gamma\|_\infty$ gives $\|\nabla\rho_\gamma \|_{L_2(T\calM)} \to 0$.
\end{proof}

We are now ready to establish equivalence between weak solutions of the stationary McKean-Vlasov equation from Definition~\ref{def:weak:mv}, the zeros of the Gibbs map \eqref{eq:gibbs-map} and critical points of~$\calF_\gamma$.
\begin{proposition}
	\label{prop:equivalence}
	For $\rho\in H^1(\calM) \cap \calP_{ac}^+(\calM)$ and $\gamma \in \bbR_+$ the following statements are equivalent:
	\begin{enumerate}
		\item $\rho$ is a weak solution of the stationary McKean-Vlasov equation \eqref{eq:mckean-vlasov}  in the sense of Definition~\ref{def:weak:mv},
		\item $(\rho, \gamma)$ is a solution of $ F(\rho, \gamma) = 0$, where  $F$ is the Gibbs map defined in \eqref{eq:gibbs-map}.
		\item $\rho$ is a critical point of the free energy functional $\calF_\gamma$ \eqref{eq:free-energy}.
	\end{enumerate}
\end{proposition}
\begin{proof}
	\textbf{(2)$\to$(1)} Let $\rho \in L_1(\calM)$ be a solution of $F(\rho, \gamma) = 0$. By Lemma \ref{lemma:gibbs-H1}, $\rho \in H^1(\calM)$ and by differentiating  $F(\rho, \gamma)$ we obtain 
	\[
	\nabla F(\rho, \gamma) = \nabla \rho -\gamma\frac{e^{-\gamma (W *\rho)(x)}}{Z(\rho, \gamma)}\nabla_x W(x, \cdot) * \rho =\nabla \rho -\gamma \rho \nabla_x W(x, \cdot) * \rho = 0.
	\]
	Testing against $\psi \in L_2(T\calM)$ shows that $\rho$ is a weak solution of McKean-Vlasov equation.
	
	\textbf{(1)$\to$(2)} Let $\rho \in H^1(\calM)$ be a weak solution of \eqref{eq:mckean-vlasov}, then $v = \rho$ is a solution of a "frozen" linear equation
	\begin{equation}
		\label{eq:mv-frozen}
		\gamma^{-1}\int_{\calM}g(\nabla v, \nabla \phi)dm + \int_{\calM} g(v \nabla\phi, \nabla_x W(x,\cdot) *\rho) dm  =0,
	\end{equation}
	for every $\phi \in H^1(\calM)$. Let $T\psi := \frac{1}{Z(\gamma, \psi)} e^{-\gamma W *\psi}$. In Lemma \ref{lemma:gibbs-H1} we have shown that $\|W*\rho\|_\infty <\infty$ and therefore $T\rho$ is uniformly bounded away from zero 
	\[
	(T\rho)(x) \geq \frac{e^{-\gamma\|W*\rho\|_\infty}}{m(\calM)e^{\gamma\|W*\rho\|_\infty}} > 0
	\]
	for any $\rho \in L_1(\calM)\cap \calP_{ac}(\calM)$.
	Consider the change of variables $h(x) = v(x)/(T\rho)(x)$ and note that $h$ satisfies
	\[
	\nabla v(x) = (T\rho)(x)\nabla h(x) + h(x)\nabla(T\rho)(x).
	\]
	Using the fact that $\nabla(T\rho)(x) =-\gamma (T\rho)(x)(\nabla_xW(x,\cdot)*\rho)(x)$ one can see that \eqref{eq:mv-frozen} for any $\phi \in H^1(\calM)$ rewrites as
	\begin{equation}
		\label{eq:elliptic-PDE}
		\int_{\calM} g(\nabla\phi,  T\rho \nabla h) dm  =0. 
	\end{equation}
	Recall from the proof of Lemma \ref{lemma:gibbs-H1} that $\|T\rho \|_\infty <\infty$ and thus \eqref{eq:elliptic-PDE} is weak formulation of a uniform-elliptic PDE 
	\[
	-\divr(T\rho\nabla h)=0.
	\]
	Similar to the Euclidean case, the only solutions satisfy $\nabla h = 0$ in $L_2(T\calM)$ sense and thus are constant functions $h = const$. By definition of $h$ we obtain for some $c>0$ that 
	\[
	\rho = v = c \; T\rho\,.
	\]
	and since $\|T\rho\|_{L_1} = 1$ we conclude that the only solution is $\rho = T\rho$.
	
	\textbf{(2)$\to$(3)} Let $\rho$ be a zero of the Gibbs map, take arbitrary $\rho' \in \calP_{ac}(\calM)$ and consider the curve $\rho_s = s\rho' + (1-s)\rho$ for $s\in[0,1]$. Applying $\calF_\gamma$ to $\rho_s$ and differentiating with respect to $s$ we obtain
	\[
	\frac{d}{ds}\calF_\gamma(\rho_s)\Big|_{s=0} = \int_\calM \left(\gamma^{-1}\log \rho + W*\rho \right)(\rho' - \rho)dm.
	\]
	Since $\rho$ is a zero of the Gibbs map we know that $\rho =  \frac{1}{Z(\gamma, \rho)} e^{-\gamma (W *\rho)(x)}$ and thus the above integral takes the form
	\begin{equation}
		\label{eq:2to3}
		\int_\calM \left(\gamma^{-1}\log \rho + W*\rho \right)(\rho' - \rho)dm= -\int_\calM \gamma^{-1}\log Z(\gamma, \rho) (\rho' - \rho)dm =0,
	\end{equation}
	so $\rho$ is a critical point of $\calF_\gamma$.
	
	\textbf{(3)$\to$(2)} Since $\rho \in H^1$, there exists a gradient of $\rho$ almost everywhere and thus it is almost everywhere continuous. Take an arbitrary point of continuity $x_0 \in \calM$, we show that 
	\[
	\gamma^{-1}\log \rho (x_0) + (W*\rho)(x_0) = \frac{1}{m(\calM)}\int_\calM \bigl(\gamma^{-1}\log \rho + W*\rho \bigr)dm = \text{const.} \, .
	\]
    First assume that there exist $\alpha_0 >0$ such that $\rho(x) \geq \alpha_0$ and we can take a sequence of positive densities $(\rho_n')_{n\in\bbN}$ of the form
	\[
	\rho'_n(x) = \begin{cases}
		\rho(x) + \frac{\alpha_0}{m(B(x_0, 1/(n +R)))} \qquad &\text{if } x\in B(x_0, 1/(n+R)), \\
		\rho(x) - \frac{\alpha_0}{m(\calM)- m(B(x_0, 1/(n+R)))}\qquad &\text{otherwise,} 
	\end{cases}
	\]
	for some $R >0$. Then from \eqref{eq:2to3} we obtain
	\begin{align}
		\MoveEqLeft\frac{\alpha_0}{m(B(x_0, 1/(n +R)))}\int_{B(x_0, 1/(n+R))} \left(\gamma^{-1}\log \rho + W*\rho \right)dm \label{eq:3to2-left}\\
		&= \frac{\alpha_0}{m(\calM)- m(B(x_0, 1/(n+R)))}\int_{\calM\backslash B(x_0, 1/(n+R))} \left(\gamma^{-1}\log \rho + W*\rho \right)dm.\label{eq:3to2-right}
	\end{align}
	Since $x_0$ is a point of continuity, the limit of the \eqref{eq:3to2-left} is simply the point evaluation
	\[
	\lim_{n\to \infty}\frac{\alpha_0}{m(B(x_0, 1/(n +R)))}\int_{B(x_0, 1/(n+R))} 
	\mkern-20mu \left(\gamma^{-1}\log \rho + W*\rho \right)dm = \bigl(\alpha_0\gamma^{-1}\log \rho + (W*\rho)\bigr)(x_0),
	\]
	and by the same argument the right hand side \eqref{eq:3to2-right} equals to the integral with respect to the volume measure
	\begin{align*}
		\MoveEqLeft\lim_{n\to \infty}\frac{\alpha_0}{m(\calM)- m(B(x_0, 1/(n+R)))}\int_{\calM\backslash B(x_0, 1/(n+R))} \left(\gamma^{-1}\log \rho + W*\rho \right)dm\\
		&= \alpha_0\int_{\calM}\left(\gamma^{-1}\log \rho + (W*\rho)\right)dm.    
	\end{align*}
	As a result we conclude that $\gamma^{-1}\log \rho + (W*\rho) = \text{const.}$\@ $m$-almost everywhere, and since $\rho$ is a probability measure we get the scaling
	\[
	\rho = \frac{1}{Z(\gamma, \rho)}e^{-\gamma(W*\rho)}.
	\]
	If $\rho$ is not bounded away from zero, we can choose an arbitrary small $\alpha_\varepsilon \in \bbR_+$ and show that the expression $\gamma^{-1}\log \rho + W*\rho$ is constant on every set of form $A_{\varepsilon} := \{x\in \calM: \rho(x) \geq \alpha_\varepsilon\}$. Since $\alpha_\varepsilon$ is arbitrary, we get the result.
\end{proof}
\begin{remark}
	Proposition~\ref{prop:equivalence} shows that the invariant measures do not depend on the induced metric $g$ but only on the interaction kernel $W$. Because we have the formulation of solutions of \eqref{eq:mckean-vlasov} in terms of the Gibbs map, one can see that for two different parametrization of the manifold $\calM: x = x_1(\theta_1) = x_2(\theta_2)$ the sets of solutions will be identical, assuming that they induce the same volume measure $m$ and that the interaction kernel is independent of the parametrization in the sense that $W(x_1(\theta_1), y_1(\theta_1)) = W(x_2(\theta_2), y_2(\theta_2))$ for all pairs of points $x, y \in \calM$. Using the energetic interpretation of the stationary measures, one can say that an invariant measure stays invariant under any re-parametrization which does not affect the interaction between particles.
\end{remark}
Finally, using the established equivalence and the $H^1$ convergence proved in Corollary~\ref{cor:gibbs-gamma0} we see that the solutions of the stationary McKean-Vlasov equation converge to the kernel of the Laplace-Beltrami operator, consisting just of constants, in the limit of infinitely small interaction $\gamma \to 0$.
\begin{corollary}
	\label{cor:convergence-min}
	Let the sequence of parameters $(\gamma_n)_{n\in\bbN}$ be such that $\gamma_n \in \bbR_+$ and $\gamma_n \to 0$. Let $W: \calM\times\calM \to \bbR$ be a continuous $H^1$ function on $\calM\times\calM$ satisfying  $W(x,y)=W(y,x)$, then the sequence of solutions of \eqref{eq:mckean-vlasov}, if they exist, converges in $H^1$ to $\bar\rho$ 
	\[
	\rho_\gamma \stackrel{H^1}{\to} \bar \rho,
	\]
	where $\bar \rho = \frac{1}{m(\calM)}$ is the unique (up to rescaling) solution of $\Delta \rho = 0$.
\end{corollary}
We show existence of minimizers in the next section. The small noise limit $\gamma \to \infty$ is more involved since the number and the structure of the solutions of the pure interaction PDE strongly depends on the interaction potential $W$, so is is only possible to show convergence up to a subsequence. In addition, for $\gamma = \infty$ solutions of \eqref{eq:mckean-vlasov} are no longer guaranteed to be $H^1$ functions, so we are only able to show convergence in the weak sense, see Lemma \ref{prop:gamma-infty}.

\subsection{Existence of minimizers}
\label{sec:existence}
Let $m$ be a normalized volume measure such that $m(\calM) = 1$. We consider the free energy functional of form \eqref{eq:free-energy}
with continuous interaction kernel $W: \calM\times\calM \to \bbR$. We show that for arbitrary value of $\gamma \in\bbR_+$ there exist a minimizer of the free energy functional on the space of probability measures $\calP(\calM)$, the minimizer admits density, and the density is an $L_2$ function. 

\begin{theorem}
	\label{th:minimizers}
	Let $\calF_\gamma$ be as defined in \eqref{eq:free-energy} and $W: \calM\times\calM \to \bbR$ be a continuous function on $\calM\times\calM$ satisfying  $W(x,y)=W(y,x)$, then there exist at least on minimizer $\mu^*$ in the space of probability measures $\calP(\calM)$
	\[
	\mu^* \in \argmin_{\mu\in \calP(\calM)}\calF(\mu).
	\]
	Moreover, every minimizer $\mu^*$ admits density w.r.t. normalized volume measure $d\mu^* = \rho^* dm$ and the density is a square-integrable function, $\rho^* \in L_2(\calM)$.%
\end{theorem}
\begin{proof}
	As follows from the compactness of $\calM$, the interaction kernel $W$ is bounded on it's domain; we will denote it's minimum and maximum  as $W_{\min} = \min_{x, y \in \calM} W(x, y)$ and $W_{\max} = \max_{x, y \in \calM}W(x, y)$. The proof is divided in two steps, in the first step we show existence of minimizers in the space of positive measures absolutely continuous with respect to the volume measure $\calP_{ac}^+(\calM)$, where
	\[
	\calP_{ac}^+(\calM) = \set*{\mu\in \calP(\calM): d\mu = \rho dm, \ \int \rho(x)dm(x) = 1, \ \rho(x)> 0 \ m-\text{a.e.}}.
	\]
	It is easy to see that bounded interaction kernel, the interaction energy is bounded for any $\mu \in \calP(\calM)$ and the entropy is finite only on $\calP^+_{ac}(\calM)$, and thus if a minimizer $\rho^*$ exist, it is an element of $\calP_{ac}^+(\calM)$.  At the second step we show the existence of an upper bound of the minimizer $C \in \bbR_+: \ \rho(x) \leq C $ for $m$-a.e. $x$. Then it is naturally follows that $\rho^*$ is square-integrable
	\[
	\int_{\calM} \rho(x)^2 dm(x) \leq C^2\int_{\calM}  dm(x) = C^2,
	\]
	in other words, $\rho^* \in L_2(\calM)$.

	\paragraph*{Existence of minimizers:}
	Take a minimizing sequence $(\rho_n)_{n\in \bbN}$, $\rho_n \in \calP_{ac}^+(\calM)$
	\[
	\inf_{\calP_{ac}^+(\calM)}\calF(\rho) = \lim_{n\to\infty}\calF(\rho_n).
	\]
	Since $\calM$ is a compact space, every sequence in $\calP_{ac}^+(\calM) \subset \calP(\calM)$ is tight and, by Prokhorov's theorem, relatively weakly compact in $\calP(\calM)$. Take a convergent subsequence $\rho_{n_k} \stackrel{w}{\to} \rho^* \in \calP(\calM)$ of $(\rho_n)_{n\in \bbN}$. 
	
	The entropy term is a weakly lower-semicontinuous functional on the space of measures $\calP(\calM)$  (see for example \cite[Lemma 1.4.3]{dupuis2011weak}). Using \cite[Lemma 7.3]{santambrogio2015optimal} we get weak convergence of the product measures along the convergent subsequence $\rho_{n_k}$:
	\[
	\rho_{n_k} \otimes\rho_{n_k} \stackrel{w}{\to} \rho^* \otimes\rho^*.
	\]
	Using the above and the boundedness of the interaction kernel we prove the continuity of the interaction energy \eqref{eq:interaction-energy}:
	\[
	\calI(\rho_{n_k})= \int_{\calM\times\calM} \mkern-10mu W(x, y )\rho_{n_k}(x)\rho_{n_k}(y)dm(x)dm(y) \to \int_{\calM\times\calM}  \mkern-10mu  W(x, y )\rho^*(x)\rho^*(y)dm(x)dm(y).
	\]
	As a result, $\calF$ is weakly lower-semicontinuous on $\calP(\calM)$ as a sum of lower-semicontinuous functionals. Moreover, since $\calF_\gamma(\rho^*) <\infty$ we conclude that $\rho^* \in \calP_{ac}(\calM)$ and by direct method of calculus of variations 
	\[
	\calF_\gamma(\rho^*) =\argmin_{\rho \in \calP(\calM)} \calF_\gamma(\rho) = \argmin_{\rho \in \calP_{ac}^+(\calM)} \calF_\gamma(\rho). 
	\]
	\textbf{Upper bound:} 
	The construction follows a similar approach from~\cite{vollmer2018bifurcation}, where this is done on the sphere $\bbS^2$.
	Let $\rho^*$ be a minimizer of $\calF$. Let $C = \exp(12\gamma(W_{\max} - W_{\min}) +4)$ and assume that there exist set $A_{>C} := \{x\in \calM: \rho^*(x)> C\}$ of positive measure $m(A_{>C}) > 0$. Let $A_{<2} = \{x\in \calM: \rho^*(x)< 2\}$, and note that $A_{<2}$ has a positive measaure because
	\begin{align*}
		1 &= \int_{\calM}\rho^*(x)dm(x) \geq \int_{\calM \backslash A_{<2}}\rho^*(x)dm(x) \geq 2(1-m(A_{<2}))
	\end{align*}
	which after rearranging gives
	\[
	m(A_{<2}) \geq \frac{1}{2}.
	\]
	Define a density $\hat \rho^* \in \calP_{ac}^+(\calM)$:
	\[
	\hat \rho^*(x) = \begin{cases}
		C ,\quad &x\in A_{>C}, \\
		\rho^*(x), \quad &x\in \calM\backslash (A_{>C}\cup A_{<2}), \\
		\rho^*(x) + \delta, &x\in A_{<2}, 
	\end{cases}
	\]
	where $\delta =\frac{\int_{A_{>C}}(\rho^*(x) - C)dm(x)}{m(A_{<2})} \leq 2$. We will show that $\calF(\hat \rho^* ) <\calF(\rho^* ) $, implying that $\rho^*$ can not be a minimizer.
	For the entropy we have
	\begin{align*}
		\MoveEqLeft \int_{\calM}\mkern-4mu\bra*{\rho^*\log \rho^* - \hat \rho^*\log\hat \rho^*}dm = \int_{A_{>C}}\mkern-8mu\bra*{\rho^*\log \rho^* - \hat \rho^*\log\hat \rho^*}dm + \int_{A_{<1}}\mkern-8mu\bra*{\rho^*\log \rho^* - \hat \rho^*\log\hat \rho^*} dm \\
		&\geq(\log C+1)\int_{A_{>C}} (\rho^* - C)dm - \delta\int_{A_{<1}} \left(\log(\rho^* +\delta) + 1 \right)dm \\
		&\geq(\log C+1)\int_{A_{>C}} (\rho^* - C)dm - \delta m(A_{<2}) \left(\log(1 +\delta) + 1 \right) \\
        &= \delta m(A_{<2})\left(\log C - \log(1+\delta)\right) \\
		&\geq \frac12\delta \left(\log C - \log 3\right).
	\end{align*}
	And the difference of the interaction energy can be naively bounded as follows
	\begin{align}
		\MoveEqLeft \int_{\calM\times\calM}W(x, y)\rho^*(x)\rho^*(y)dm(x)dm(y) - \int_{\calM\times\calM}W(x, y)\hat \rho^*(x)\hat \rho^*(y)dm(x)dm(y) \notag \\
		&=\int_{\calM\times\calM}(W(x, y)- W_{\min})\rho^*(x)\rho^*(y)dm(x)dm(y) \notag \\
		&\qquad- \int_{\calM\times\calM}(W(x, y)- W_{\min})\hat \rho^*(x)\hat \rho^*(y)dm(x)dm(y)\notag \\
		&= \int_{A_{>C}\times A_{>C}}(W(x, y)- W_{\min})(\rho^*(x)\rho^*(y) - C^2)dm(x)dm(y) \label{eq:interact:cc}\\
		&+\int_{(\calM \backslash A_{>C})\times (\calM \backslash A_{>C})}(W(x, y)- W_{\min})(\rho^*(x)\rho^*(y) - \hat \rho^*(x)\hat \rho^*(y))dm(x)dm(y) \label{eq:interact:22}\\ 
		&+2\int_{A_{>C}\times (\calM \backslash A_{>C})}(W(x, y)- W_{\min})(\rho^*(x)\rho^*(y) - C\hat \rho^*(y))dm(x)dm(y). \label{eq:interact:2c}
	\end{align}
	The first term \eqref{eq:interact:cc} is non-negative because on the set $A_{>C}$ we have $\rho^* > C$. For the second term \eqref{eq:interact:22} we use the fact that on $\calM \backslash A_{>C}$ the difference between the densities $\rho^*, \hat\rho^*$ is bounded $\rho^* - \hat \rho^* \leq \delta$ to get the estimate:
	\begin{align*}
		\eqref{eq:interact:22} &\geq (W_{\max}-W_{\min})\int_{(\calM \backslash A_{>C})\times (\calM \backslash A_{>C})} \mkern-16mu \bigl(\rho^*(x)\rho^*(y) - (\rho^*(x)+\delta)(\rho^*(y) + \delta)\bigr)dm(x)dm(y)  \\
		&= -2\delta(W_{\max}-W_{\min})\int_{\calM \backslash A_{>C}}\left(\frac12\delta+\rho^*(x)\right)dm(x)  \\
		&\geq -2\delta(W_{\max}-W_{\min})\left(m(\calM \backslash A_{>C}) + \int_{\calM \backslash A_{>C}}\rho^*(x)dm(x)\right) \geq -4\delta(W_{\max}-W_{\min}).
	\end{align*}
	Finally, the last term \eqref{eq:interact:2c} can be estimated as
	\begin{align*}
		\eqref{eq:interact:2c} &=2\int_{A_{>C}\times A_{<2}}(W(x, y)- W_{\min})(\rho^*(x)\rho^*(y) - C\rho^*(y) - C\delta)dm(x)dm(y) \\
		&\quad +2\int_{A_{>C}\times (\calM \backslash (A_{>C}\cup A_{<2}))}(W(x, y)- W_{\min})(\rho^*(x)\rho^*(y) - C\rho^*(y))dm(x)dm(y) \\
		&\geq 2\int_{A_{>C}\times A_{<2}}(W(x, y)- W_{\min})(\rho^*(x)- C)\rho^*(y) dm(x)dm(y) \\
		&\quad -2\delta(W_{\max}- W_{\min})\int_{A_{>C}\times (\calM \backslash (A_{>C}\cup A_{<2}))} C dm(x)dm(y) \\
		&\geq 0 - 2\delta(W_{\max}- W_{\min})m\left(\calM \backslash (A_{>C}\cup A_{<2})\right)\int_{A_{>C}} C dm(x) \geq -2\delta(W_{\max}- W_{\min}).
	\end{align*}
	Combining the above estimates we conclude that
	\[
	\calF_\gamma(\rho^* ) - \calF_\gamma(\hat \rho^* ) \geq \delta\gamma^{-1} \left(\frac12\log C - \frac12\log 3\right) - 6\delta(W_{\max}-W_{\min})\geq 0,
	\]
	implying that any minimizer $\rho^*$ is uniformly bounded by $C$, which completes the proof. %
\end{proof}
\subsection{Limit of small noise}
\label{sec:large-gamma}
In this section we study the limiting behavior of the minimizers of the free energy functional~\eqref{eq:free-energy} in the small noise regime $\gamma\to \infty$. Intuitively, as the noise ratio gets smaller, the resulting PDE tends to recover the behaviour of the pure interaction system.
We consider a sequence of parameter values $(\gamma_n)_{n\in \bbN}$ with $\gamma_n \to \infty$. Since there always exist a minimizer we then consider a sequence of such minimizers $(\rho_n)_{n\in\bbN}$, where $\rho_n \in \argmin \calF_{\gamma_n}$. Using the theory of $\Gamma$-convergence (see Appendix~\ref{appendix:Gamma}) we show that all the limiting points of such a sequence are the minimizers of the interaction energy $\calI$.
\begin{proposition}
	\label{prop:gamma-infty}
	Let $\calF_\gamma$ be as defined in \eqref{eq:free-energy} and $W: \calM\times\calM \to \bbR$ be a continuous function on $\calM\times\calM$ satisfying  $W(x,y)=W(y,x)$. Let $(\gamma_n)_{n\in \bbN}$ be a positive, increasing sequence satisfying $\gamma_n \to \infty$. Let $(\rho_n)_{n\in \bbN}$ be a sequence of minimizers of $\calF_{\gamma_n}$, then there exist a weakly convergent subsequence $\rho_{n_k}$ such that $\rho_{n_k} \stackrel{w}{\to} \rho^*$ and $\rho^*$ is the minimizer of the interaction energy
	\[
	\rho^* \in \argmin_{\rho \in \calP(\calM)} \calI(\rho).
	\]
\end{proposition}
\begin{proof}
	Consider a shifted functional $\bar\calF_\gamma = \calF_\gamma - \gamma^{-1}\calE(\bar\rho)$, since the last term is a constant, minimizers of $\bar\calF_\gamma$ coincide with the minimizers of $\calF_\gamma$. At the same time for $\gamma_1 > \gamma_2 > 0$ and arbitrary $\rho \in \calP(\calM)$ we have
	\[
	\bar\calF_{\gamma_1}(\rho) = \calI(\rho) + \gamma_1^{-1}\left(\calE(\rho) - \calE(\bar\rho)\right) \leq  \calI(\rho) + \gamma_2^{-1}\left(\calE(\rho) - \calE(\bar\rho)\right) = \bar\calF_{\gamma_2}(\rho),
	\]
	so the sequence $(\bar\calF_{\gamma_n})_{n\in\bbN}$ is decreasing. At the same time, the pointwise limit of $\bar\calF_{\gamma_n}$ is
	\[
	\bar \calF =\lim_{n\to\infty}\bar\calF_{\gamma_n}(\rho) = \begin{cases}
		\calI(\rho), \qquad &\rho \in \calP_{ac}^+(\calM), \\
		+\infty &\text{otherwise.}
	\end{cases}
	\]
	By Proposition \ref{prop:gamma-decreasing} $\bar\calF_{\gamma_n} \stackrel{\Gamma}{\to} \text{lsc}(\bar \calF)$, where the lower-semicontinuous envelope of $\bar \calF$ is exactly~$\calI$. As shown in Theorem \ref{th:minimizers}, $\calI$ is a weakly lower-semicontinuous functional, so we only need to show that there exists no lower-semicontinuous functional $\calG\neq \bar\calF$ satisfying $\calI \leq \calG\leq \bar\calF$. Since $\bar\calF = \calI$ on $\calP_{ac}^+(\calM)$ we only need to consider $\rho \in \calP(\calM) \backslash \calP_{ac}^+(\calM)$. The space of measures absolutely continuous w.r.t. the volume measure $\calP_{ac}(\calM)$ is dense in $\calP(\calM)$ and by simple construction $\calP_{ac}^+(\calM)$ is dense in $\calP(\calM)$. Taking a sequence $\rho_n \stackrel{w}{\to} \rho$, where $\rho_n \in \calP_{ac}^+(\calM)$ we conclude that $\text{lsc}(\bar\calF)(\rho) \leq \calI(\rho)$ and thus $\text{lsc}(\bar\calF) = \calI$. Applying the fundamental theorem of $\Gamma$-convergence (Theorem \ref{th:gamma-coonvergence}) we get the result.
\end{proof}
\begin{remark}[Limitations]
	Note that for the small noise limit we only show convergence of the minimizers of the free energy functional, while the stationary solutions of the McKean-Vlasov equations are all of the critical points. We also do not answer the reverse question, namely whether every minimizer of the interaction energy can be approximated by the minimizers of the free energy functional with (infinitely)-large $\gamma$. %
\end{remark}

\subsection{Geodesic convexity}
\label{sec:convexity}

In this section we use the approach adapted from \cite{sturm2005convex} to characterize the convexity of the free energy functional \eqref{eq:free-energy}. The idea of generalizing the convexity criterion for the interaction potential on $\bbR^d$ to the manifold setting has been discussed in \cite[Chapter 17]{Villani2008}, but since we found no rigorous formulation in the literature we prove such a criterion in this Section. 
With a slight abuse of notation we will use $\calE(\rho)$ instead of $\calE(\mu)$ if $\mu$ admits density $\rho$.

A functional is geodesically convex if it satisfies the following definition.
\begin{definition}[Geodesic convexity]
	A functional $F: \calX \to \bbR$ on a metric space $(\calX, d)$ is geodesically $\lambda$-convex for $\lambda\in \bbR$ if for any geodesic $\gamma: [0,1] \to \calX$ it holds that
	\[
	F(\gamma(s)) \leq (1-s)F(\gamma(0)) + sF(\gamma(1)) -\frac{\lambda}{2} s(1-s) d(\gamma(0), \gamma(1)), \quad \forall s\in [0,1].
	\]
\end{definition}
For a lower-semicontinuous function $f:[0,1] \to \bbR$ define the lower centered second derivative 
\[
\underline{\partial_t^2}f(t) = \lim\inf_{s\to 0} \frac1{s^2}\left[f(t+s)+ f(t-s) - 2f(t)\right].
\]
Then a functional is $\lambda$-convex if and only if it is lower semicontinuous along geodesics and if for each geodesic $\gamma:[0,1] \to \calX$ with $F(\gamma(0)), F(\gamma(1)) < \infty$, it holds that $ F(\gamma(s)) \leq \infty$ for all $s\in (0,1)$ and
\[
\underline{\partial_s^2}F(\gamma(s))  \geq \lambda d(\gamma(0), \gamma(1))^2.
\]

We give a sufficient condition for $\lambda$-convexity of the functional \eqref{eq:free-energy} on the space of probability measures on a Riemannian manifold $\calM$ with finite second moment 
\[
\calP_2(\calM) := \{\mu \in \calP(\calM): \int \dist(x, x_0)^2d\mu <\infty\},
\]
for some $x_0 \in \calM$, equipped with Wasserstein metric $\fw_2$. For any two measures $\mu, \nu \in \calP_2(\calM)$ the $\fw_2$ distance is 
\[
\fw_2(\mu, \nu) := \inf_{\pi \in \Pi(\mu, \nu)}\left(\int \dist(x, y)^2d\pi(x, y)\right)^{1/2},
\]
where infimum is taken with respect to all possible couplings $\pi$ with first and second marginals being $\mu$ and $\nu$ respectively. Note that since $\calM$ is compact $\calP(\calM) = \calP_2(\calM)$, we continue using $\calP_2$ in this section to emphasise the usage of the Wasserstein-2 topology.

We begin by stating some relevant results.

\begin{lemma}[Lemma 3.1 \cite{sturm2005convex}] Let $\mu_0, \mu_1 \in \calP_2(\calM)$ admit densities $\rho_1, \rho_2 > 0$ w.r.t. the volume measure $m$. Then there exists a unique geodesic $\mu: [0,1] \to \calP_2(\calM)$ such that $\mu(0) = \mu_0, \ \mu(1) = \mu_1$ and for all $s \in [0,1]$ $\mu(s)$ is absolutely continuous w.r.t. $m$. Moreover, there exists a vector field $\Phi:\calM \to T\calM$ such that $\mu(s)$ is the push forward of $\mu_0$ under the map
	\[
	F_s: \calM \to \calM \quad\text{with} \quad F_s(x)=\exp_x(s\Phi).
	\]    
\end{lemma}
Note that by definition of the push forward the above implies that for any measurable function $u:\calM\to \R$ it holds that
\[
\int_\calM u(x)d\mu_s(x) = \int_\calM u(F_s(x))d\mu_0(x).
\]
\begin{lemma}[Corollary 1.5 \cite{sturm2005convex}]
	\label{lemma:entropy-convexity}
	Consider the entropy $\calE$  defined in \eqref{eq:entropy}. Then the lower second derivative of $\calE$ along geodesic $\rho_t$, with $\calE(\rho_0), \calE(\rho_1) < \infty$, satisfies 
	\[    \underline{\partial_t^2}\calE = \int 
	\operatorname{Ric}_x(\dot{F_t}, \dot{F_t})\rho_0(x)dm(x) 
	\]
	Moreover, $\calE$ is $\lambda$-convex for $\lambda\in\R$ if and only if $\forall x \in \calM, \ v\in T_x\calM$
	\[
	\operatorname{Ric}_x(v, v) \geq \lambda\|v\|^2.
	\]
\end{lemma}
Extending the result to the free energy functional $\calF_\gamma$ with the interaction term \eqref{eq:free-energy} we get the following sufficient condition for the geodesic convexity of $\calF_\gamma$.
\begin{theorem}
	\label{th:convexity-M}
	Consider the free energy $\calF_\gamma$ as defined in \eqref{eq:free-energy}. Assume that there exist $\alpha, \lambda \in \bbR$ such that $W$ satisfies
	\[
	\underline{\partial^2_t} W\left(\exp_x vt, \exp_y ut\right) \geq \alpha(\|v\|^2 + \|u\|^2)
	\]
	and $\calM$ is such that
	\[
	\operatorname{Ric}_x(v, v) \geq \lambda\|v\|^2
	\]
	for all $x, y \in \calM, \ v\in T_x\calM, u \in T_y\calM$,
	then $\calF_\gamma$ is an $(\gamma^{-1}\lambda + \alpha)$-convex functional. In particular, if $\underline{\partial^2_t} W\left(\exp_x vt, \exp_y ut\right) \geq 0$, $\calF_\gamma$ is $\gamma^{-1}\lambda$-convex.
\end{theorem}
\begin{proof}
	Recall that \eqref{eq:free-energy} is a sum of entropy and interaction energy $\calF = \gamma^{-1}\calE + \calI$. By definition of the lower second derivative we get
	\[
	\underline{\partial_t^2}\calF \geq \gamma^{-1}\underline{\partial_t^2}\calE + \underline{\partial_t^2}\calI.
	\]
	Let $\rho_t$ be a geodesic with boundary values satisfying $\calE(\rho_0), \calE(\rho_1) < \infty$. We calculate the lower second derivative of the interaction energy along $\rho_t$. We begin by rewriting the interaction energy in term of the map $F_t$, namely
	\[
	\calI(\mu_t) = \frac{1}{2}\int_{\calM \times\calM} W(x, y )d\mu_t(x)d\mu_t(y) =  \frac{1}{2}\int_{\calM \times\calM} W(F_t(x), F_t(y) )d\mu_0(x)d\mu_0(y).
	\]
	Then by definition of the lower second derivative we get
	\begin{align*}   \underline{\partial_t^2}\calI  &= \lim\inf_{s\to 0} \frac1{s^2}\left[f(t+s)+ f(t-s) - 2f(t)\right] \\
		&=\lim\inf_{s\to 0}\frac1{s^2}\int_{\calM \times\calM}\Big[W(F_{t+s}(x), F_{t+s}(y)) + W(F_{t-s}(x), F_{t-s}(y)) \\
		&\hspace{110pt}-2W(F_t(x), F_t(y))\Big]d\mu_0(x)d\mu_0(y) \\
		&\geq \int_{\calM \times\calM} \underline{\partial_t^2} W(F_t(x), F_t(y))d\mu_0(x)d\mu_0(y) \\
		&\geq \alpha \int_{\calM \times\calM}  \left( \|\dot{F}_t(x)\|^2+ \|\dot{F}_t(y)\|^2\right)d\mu_0(x)d\mu_0(y) = 2\alpha\int_{\calM}\|\dot{F}_0\|d\mu_0 = 2\alpha \fw_2^2(\mu_0, \mu_1).
	\end{align*}
	Combining the estimate with the bound from Lemma \ref{lemma:entropy-convexity} we get the result.
\end{proof}
\begin{remark}
	In the Euclidean case, $\calM = \bbR^d$, the criterion from Theorem \ref{th:convexity-M} reduces to $\alpha$-convexity of the interaction kernel $W: \bbR^{2d} \to \bbR$. As remarked in \cite[Proposition 7.25]{santambrogio2015optimal}, it is a sufficient but not necessary condition for the convexity of the corresponding interaction potential $S$.
\end{remark}
\begin{remark}[Gradient flow solutions]
	Formally, from the convexity properties one can also deduce existence (and uniqueness in case of $\lambda>0$) of a \emph{gradient flow solution} of the corresponding non-stationary McKean-Vlasov equation. For a separable Hilbert space $X$, such result for a large class of functionals on Wasserstein space $\calP_2(X)$ is rigorously established in \cite[Section 11.2]{ambrosio2005gradient}. On a manifold of positive curvature similar result was proved for the relative entropy (without the interaction term) in \cite{erbar2010heat}. 
    \end{remark}
    \begin{remark}[Functional inequalities]
	In Euclidean space the uniform geodesic convexity has been shown to be equivalent to the log-Sobolev inequality \cite{Villani2003}. We expect the same arguments to hold on smooth manifolds. On the equivalence of functional inequalities in Riemannian setting see \cite{otto2000generalization}. Logarithmic Sobolev inequality in the special case $\calM = \bbS^{n-1}$ is studied in \cite{brigati2023logarithmic}
\end{remark}

\paragraph*{The case of the sphere $\calM = \bbS^{n-1}$}
Consider a special case, namely $\calM = \bbS^{n-1}$. Note that any element of a unit sphere $x\in \bbS^{n-1}$ can be identified with a unit vector in $\bbR^{n}$. For any pair of points on a sphere $x, y \in \bbS^{n-1}$ we denote by $\left<x, y\right>$ a Euclidean scalar product between the corresponding vectors in $\bbR^n$. We now establish a sufficient condition for a convexity of an interaction energy for an interaction potential that defined in terms of the scalar product $W(x, y) = W(\left<x, y\right>)$ with now $W:[-1,1]\to\R$ by an abuse of notation.
\begin{remark}[Choice of parametrization]
	For a general manifold $\calM$ a natural choice for introducing the interaction potential is in terms of the squared geodesic distance (cf.~\cite{fetecau2021well})
	\[
	W(x, y) = W(\dist(x,y)^2).
	\]
	This choice is inconvenient in the case of a sphere, where geodesic distance is equal to
	\[
	\dist(x,y) = \arccos(\left<x, y\right>).
	\]
	The examples later are directly parametrized in terms of $\skp{x,y}$. Also, one can see that $\arccos$ is not differentiable at $\pm 1$ and in using the scalar product $\skp{x,y}$, we avoid dealing with regularity issues of the distance function at the endpoints.
\end{remark}

\begin{theorem}
	\label{th:convexity-sph}
	Consider the free energy functional $\calF_\gamma$ as defined in \eqref{eq:free-energy} on an $n$-dimensional sphere $\bbS^{n-1}$. Let the interaction kernel satisfy Assumption \ref{assum:sym-kernel} with some $W \in C^2((-1,1), \bbR)$ and let $\|W'\|_\infty, \|W''\|_\infty \leq C$. In addition let $W'(\pm 1)$ to be left/right derivative at $\pm 1$ respectively and assume that $|W'(\pm 1)|<C$, then $\calF$ is $\lambda$-convex, where $\lambda = \gamma^{-1}(n-2)-4C$.
\end{theorem}

\begin{proof}
	For any $x \in \bbS^{n-1}, \ v\in T_x\bbS^{n-1}$ let $\gamma_{x, v}: [0,1] \to \bbS^{n}$ be a constant speed geodesic with $\gamma_{x, v}(0) = x$ and $\dot\gamma_{x, v}(0) = v$. Introduce $g(x, y; v, u, t) = \left<\gamma_{x,v}(t), \gamma_{y,u}(t)\right>$, then  it's first and second derivative with respect to $t$ are:
	\begin{align*}
		\partial_t g(x, y; v, u, t) 
		&=\left<\gamma_{x,v}(t), \dot\gamma_{y,u}(t)\right> + \left<\dot \gamma_{x,v}(t), \gamma_{y,u}(t)\right> \\
		\partial^2_t g(x, y; v, u; t) 
		&=2\left<\dot\gamma_{x,v}(t), \dot\gamma_{y,u}(t)\right> + \left<\ddot \gamma_{x,v}(t), \gamma_{y,u}(t)\right> + \left<\gamma_{x,v}(t), \ddot\gamma_{y,u}(t)\right>
	\end{align*}
	As $W$ is twice differentiable on $(-1, 1)$, for $g(x, y; v, u, t) \neq \pm 1$ by the chain rule we get
	\begin{align*}      \MoveEqLeft \underline{\partial^2_t} W\left(\gamma_{x,v}(t), \gamma_{y, u}(t)\right) = \partial^2_t W(g(x, y; u, v; t)) = W''(g(\cdot; t))\left(\partial_tg(\cdot; t)\right)^2 +  W'g(\cdot; t)\partial^2_tg(\cdot; t). 
	\end{align*}
	Using that for constant speed geodesics on a sphere $\|\dot\gamma_{x,v}(t)\| = \|v\|$ and $\|\ddot\gamma_{x,v}(t)\| = \|v\|^2$ we obtain the bound
	\begin{align*}
		\underline{\partial^2_t} W\left(\gamma_{x,v}(t), \gamma_{y, u}(t)\right) \geq -C\left((\|u\| +\|v\|)^2 + \|u\|^2 + \|v\|^2\right) \geq -3C\left(\|u\|^2 +\|v\|)^2\right).
	\end{align*}
	Note that $g(x, y; v, u, t) =1$ if and only if $\gamma_{x, v}(t) = \gamma_{y, u}(t)$, so for the case $g(x, y; v, u, t) =1$ we calculate
	\begin{align*}
		\MoveEqLeft \underline{\partial^2_t} W\left(\gamma_{x,v}(t), \gamma_{y, u}(t)\right) = \lim\inf_{s\to 0} \frac1{s^2}\Big[W(\left<\gamma_{x,v}(t+s), \gamma_{y, u}(t+s)\right>) \\
		&\qquad+ W(\left<\gamma_{x,v}(t-s), \gamma_{y, u}(t-s)\right>) 
		- 2W(\left<\gamma_{x,v}(t), \gamma_{y, u}(t)\right>)\Big] \\
		&=\lim\inf_{s\to 0} \frac1{s^2}\Big[W\left(1 -s^2\|\dot\gamma_{x,v}(t)-\dot\gamma_{y,u}(t)\|^2 + o(s^2)\right)  \\
		&\qquad+ W\left(1 -s^2\|\dot\gamma_{x,v}(t)-\dot\gamma_{y,u}(t)\|^2+ o(s^2) \right) - 2W(1)\Big] \\
		&=\lim\inf_{s\to 0} \frac1{s^2}\Bigl[ 2W(1)  - 2s^2W'(1)\|\dot\gamma_{x,v}(t)-\dot\gamma_{y,u}(t)\|^2 - 2W(1) +o(s^2)\Bigr] \\
		&= -2W'(1)\|\dot\gamma_{x,v}(t)-\dot\gamma_{y,u}(t)\|^2.
	\end{align*}
	Estimating the above expression we conclude that for $g(x, y; v, u, t) =1$:
	\[\begin{split}
	\underline{\partial^2_t} W\left(\gamma_{x,v}(t), \gamma_{y, u}(t)\right) &\geq -4|W'(1)|\left(\|\dot\gamma_{x,v}(t)\|^2+\|\dot\gamma_{y,u}(t)\|^2\right) \\
	&\geq -4C\left(\|\dot\gamma_{x,v}(t)\|^2 + \|\dot\gamma_{y,u}(t)\|^2\right).
\end{split}
	\]
	Analogous result holds for $g(x, y; v, u, t) =-1$. As a result, the assumption of Theorem~\ref{th:convexity-M} is satisfied for $\alpha = -4C$.
\end{proof}
Note that if $\lambda >0$ the above Theorems~\ref{th:minimizers} and~\ref{th:convexity-sph} guarantee strict convexity of the free-energy functional, and by the properties of strictly convex functionals we can conclude that $\calF_\gamma$ has a unique minimizer.
\begin{corollary}
	\label{cor:convexity-sph}
	Consider free energy functional $\calF_\gamma$ with the interaction kernel $W$ as in Theorem \ref{th:convexity-sph}. Let $\gamma < \frac{(n-2)}{4C}$, then there exists a unique minimizer of $\calF_\gamma$.
\end{corollary}
\begin{proof}
	The statement follows directly from Theorems~\ref{th:minimizers} and~\ref{th:convexity-sph}.
\end{proof}

\section{The structure of \texorpdfstring{$L_2(\bbS^{n-1})$}{the Lebesgue space on the sphere}}

We collect some properties of spherical harmonics in Section~\ref{sec:harmonics}, the convolution and interaction on a sphere in Sections~\ref{sec:spherical-conv} and~\ref{ssec:InteractionSphere}, respectively.

\subsection{Spherical harmonics}
\label{sec:harmonics}
In our further analysis we will imploy the Hilbertian structure of $L_2(\bbS^{n-1})$. For that we will need some results concerning spherical harmonics, an orthonormal basis of square-integrable functions on a sphere. In this section we introduce the spherical harmonics basis and relate it to the eigenspaces of the Laplace-Beltrami operator. In the next section we give a definition of the convolution on a sphere and introduce the analog of the convolution theorem for $L_2(\bbS^{n-1})$. For more details on the topic we refer the reader to \cite[Chapters 1 and 2]{dai2013approximation}. 

    We will work with the Hilbert space of square-integrable functions on $\bbS^{n-1}$ equipped with the scalar product
    \[
    \left<f, g\right> = \frac{1}{\omega_n}
    \int_{\bbS^{n-1}}f(x)g(x)d\sigma(x),
    \]
    where $\sigma$ is the spherical measure (a uniform measure on $\bbS^{n-1}$) and  $\omega_n = \sigma(\bbS^{n-1}) = \frac{2\pi^{n/2}}{\Gamma(n/2)}$ is the normalization constant. We will also refer to the normalized version of the spherical measure as the uniform measure or uniform state $\bar\rho = \frac1{\omega_n}\sigma$. Spherical measure is induced by the metric generated by the standard Euclidean product in $\bbR^{n}$.  It is easy to verify that $\sigma$ is invariant under the rotation group: for any set $A \subseteq \bbS^{n-1}$ and any $g\in SO(n)$ it holds that $\sigma(gA) = \sigma(A)$, where $gA = \{ga: a\in A\}$. %

    The basis of $L_2(\bbS^{n-1})$ can be constructed with the spherical harmonics.
    \begin{definition}[$l$-th spherical harmonics $\calH_l$] For $l \in \bbN_0$ let $H_l \subset \ker \Delta$ be a linear space of $l$-homogeneous polynomials $p: \bbR^n \to \bbR$
    \[
    H_l = \{p(x): \Delta p =0, \ p(\lambda x) = \lambda^lp(x)\}.
    \]
    The linear space of $l$-th spherical harmonics $\calH_l$ is the space of the elements of $H_l$ restricted to the unit sphere:
    \[
    \calH_l = \{\tilde p(x): \exists p \in H_l \ \text{s.t.} \ \forall x \in \bbS^{n-1} \ \tilde p(x) = p(x)\},
    \]
    where $\tilde p:\bbS^{n-1} \to \bbR$.
    \end{definition}
    It can be verified that the subspaces $\calH_l,  \ \calH_m$ are orthogonal for $l\neq m$ (see \cite[Theorem 1.1.2]{dai2013approximation}) and every subspace is finite-dimensional (see \cite[Corollary 1.1.4]{dai2013approximation}). Moreover, all elements of the same subspace $\calH_l$ can be shown to be eigenfunctions of the Laplace-Beltrami operator corresponding to the same eigenvalue.
    \begin{theorem}[$\calH_l$ are eigenfunctions of $\Delta_{\bbS^{n-1}}$ {\cite[Theorem 1.4.5]
    {dai2013approximation}}] 
    \label{th:sph-egenfunctions}
    Let $\Delta_{\bbS^{n-1}}$ be the Laplace-Beltrami operator on $\bbS^{n-1}$, then for any $p \in \calH_l$ it holds
    \[
    \Delta_{\bbS^{n-1}} p = -l(n+l-2)p.
    \]      
    \end{theorem}

    An orthonormal basis of $\calH_l$ has an explicit form in spherical coordinates $\theta_1...\theta_{n-1}$:
    \begin{align*}
        x_1 &= r\sin \theta_{n-1} ...\sin\theta_{2}\sin\theta_{1}, \\
        x_2 &= r\sin \theta_{n-1} ... \sin\theta_{2}\cos\theta_{1}, \\
        &\vdots\\
        x_n &= r\cos\theta_{n-1}.
    \end{align*}
    The spherical measure in spherical coordinates takes the form
    \[
    d\sigma = %
    \prod_{i=1}^{n-1} (\sin\theta_{n-i})^{n-i-1}d\theta_{n-i}.
    \]
And the corresponding basis of spherical harmonics then has the form:
\begin{equation}
\label{eq:harmonics}
Y_{l, \bbK}(\theta) = e^{ik_{n-2}\theta_1}A_K^l \prod_{j=0}^{n-3} C_{k_j-k_{j+1}}^{\frac{n-j-2}{2}+k_{j+1}}(\cos \theta_{n-j-1}) (\sin \theta_{n-j-1})^{k_{j+1}},
\end{equation}
where $K\in \bbK_l$ is a multi-index satisfying
\begin{equation}\label{eq:MultiIndex:admissible}
\bbK_l := \set*{K = (k_0, k_1 ... k_{n-2})\in \bbN_0^{n-2}\times \bbZ: l\equiv k_0 \geq k_1 \geq \cdots \geq k_{n-3} \geq |k_{n-2}|\geq 0},
\end{equation}
with $A_K^l$ is a normalization constant and $C^\lambda_n$ are Gegenbauer polynomials of degree $n$. Gegenbauer polynomials satisfy the following recurrence relation:
\begin{equation}
\label{eq:gegenbauer}
(n+2)C_{n+2}^\lambda(t) = 2(\lambda+n+1)tC_{n+1}^\lambda(t) -(2\lambda+n)C_{n}^\lambda(t),
\end{equation}
and the first two polynomials are given by $C_0^\lambda(t) = 1$ and $C_1^\lambda(t) = 2\lambda t$.

    It can be shown that the $Y_{l, K}$ form an orthonormal basis of square-integrable function on a sphere. For any $l \in \bbN_0$ define projection of $L_2(\bbS^{n-1})$ onto $\calH_l$, such that for any $f\in L_2(\bbS^{n-1})$:
\[
\proj_l f := \sum_{K \in \bbK_l}f_{l,K} Y_{l,K}, \quad\text{with}\quad f_{l,K} := \left<f, Y_{l, K}\right>_{L_2(\bbS^{n-1})},
\]
then the following theorem holds.
\begin{theorem}[Fourier decomposition on $\bbS^{n-1}$ {\cite[Theorem 2.2.2]
{dai2013approximation}}]
\label{th:sph-decomposition}
Let $Y_{l, K}$ be the orthogonal basis of $\calH_l$ as defined in \eqref{eq:harmonics}, then the set
\[
\calY = \{ Y_{l,k}: l \in \bbN_0, k \in \bbK\}
\]
is an orthonormal basis of $L_2(\bbS^{n-1})$. In particular for any $f \in L_2(\bbS^{n-1})$ the following identity holds
\[
f =\sum_{l\in \bbN_0} \proj_l f,
\]
in the sense that $\lim_{n\to \infty} \|f -  \sum_{l=1}^n \proj_l f\|_{L_2} =0$.
\end{theorem}
Combining Theorem \ref{th:sph-decomposition} with Theorem \ref{th:sph-egenfunctions} one can see that $\calH_l$ coincides with $l$-th eigensubspace of the Laplace-Beltrami operator $\Delta_{\bbS^{n-1}}$.

We will further use the spherical harmonics basis functions $Y_{l,K}$ which are explicitly defined for a fixed coordinate system. Note that since the choice of the coordinate system is arbitrary, the results will also hold for an arbitrary spherical harmonics basis. Given that, for simplicity we fix the coordinate system for the rest of the paper.

\subsection{Spherical convolution}
\label{sec:spherical-conv}
Before we proceed to the convolution theorem on a sphere, we give some intuition on how it is related to the well-known convolution theorem in Euclidean space. Consider $L_2([-\pi, \pi]^d)$, the space of  square-integrable functions on $[-\pi, \pi]$ and 
any $f\in L_2([-\pi, \pi]^d)$ is identified with its $2\pi$-periodic extension to $\bbR^d$.
Then the set of Fourier harmonics $(w_k)_{k\in \bbZ^d}$ for $k = (k_1, ... k_d), \ k_i \in \bbZ$ defined as
\[
w_k(x) = \alpha_k\prod_{i=1}^d w_{k_i}(x_i),
\]
where $w_{k_i}: [-\pi, \pi] \to \bbR$ takes the form
\[
w_{k_i}(t) = \begin{cases} 
    \cos(k_i t), &\quad k_i > 0, \\
    1, &\quad k_i = 0, \\
    \sin(k_i t), &\quad k_i < 0,
\end{cases}
\]
and $\alpha_k$ is the normalization constant, is an orthonormal basis of $L_2([-\pi, \pi])$. For any function $f \in L_2([-\pi, \pi]^d)$ denote by $\hat f(k)$ the $k$-th component of it's Fourier decomposition:
\[
\hat f(k) := \left<f, w_k\right> = \int_{[-\pi, \pi]^d} f(x)w_k(x)dx.
\]
For any two functions $f, g \in L_2([-\pi, \pi]^d$ define their convolution as
\begin{equation}
   \label{eq:conv-rd} 
(f * g)(x) = \int_{[-\pi, \pi]^d} f(y)  g(x-y) dy,
\end{equation}
then by the convolution theorem it holds that
\[
\widehat{f*g}(k) = \hat f(k) \cdot \hat g(k). 
\]
In particular, the above representation implies that the convolution in a Euclidean space is a symmetric operator. Namely if the sum $\sum_{k\in \bbZ^d}\widehat{f*g}(k)$ is well-defined, it holds that $f*g = g*f$. 

\begin{remark}
Note that the definition \eqref{eq:conv-rd} requires the domain of the function $g$ to be a vector space, so it cannot be directly applied to $L_2(\bbS^{n-1})$. At the same time the function $h_y(x) = x-y$ can be interpreted as a shift operator, which allows to formally rewrite \eqref{eq:conv-rd} as
\[
(f * g)(x) = \int f(y) g(h_y(x)) dh_y,
\]
so that the integral is defined with respect to the measure $dh_y$ over the group of shifts. This definition can be generalized to a general group. In particular, convolution on a sphere can be defined as an integration over the special orthogonal group \cite{dokmanic2009convolution}, which can be formally introduced as:
\begin{equation}
   \label{eq:conv-rd2} 
(f * g)(x) = \int_{SO(n)} f(h)  g(hx) dh,
\end{equation}
where $dh$ is a measure on $SO(n)$ and $f$ is defined on $SO(n)$ instead of $\bbS^{n-1}$.
\end{remark}

We will now define a convolution on the sphere. Similar to the approach of \cite{dokmanic2009convolution} we will change the domain of the second argument $g$ making convolution a non-symmetric operation. We identify any element of $\bbS^{n-1}$ with the corresponding unit-length vector in $\bbR^d$. It allows us to define the convolution in a slightly more restrictive setting.
\begin{definition}[Convolution on $\bbS^{n-1}$]
	\label{def:conv-sn}
	Let $f\in L_2(\bbS^{n-1})$ and let $g: [-1, 1] \to \bbR$ satisfy
	\begin{equation}\label{eq:spherical:intergrability}
		\int_{[-1, 1]} |g(t)|(1 - t^2)^{\frac{n-3}{2}}dt < \infty,
	\end{equation}
	then the convolution of $f$ with $g$ is defined as
	\[
	(g*f)(x) :=%
	\int_{\bbS^{n-1}} f(y)g(\left<x, y\right>)d\sigma(y).
	\]
\end{definition}
We note for $g:[-1,1]\to\R$ satisfying the integrability condition~\eqref{eq:spherical:intergrability} the function $g(\skp{x_0,\cdot})\in L_1(\bbS^{n-1})$ for any $x_0\in \bbS^{n-1}$.
For such admissible $g:[-1,1]\to\R$, we define its spherical harmonic decomposition as follows:
\begin{definition}[Spherical harmonics decomposition]
	\label{def:spherical-decomposition}
	Consider $\calM = \bbS^{n-1}$, let $g: [-1, 1] \to \bbR$ satisfy the integrability condition~\eqref{eq:spherical:intergrability}
	then the sequence $(\hat g_k)_{k\in \bbN}$ is called a \emph{spherical harmonics decomposition} of $g$, where
	\[
	\hat g_k = g_{k, 0} = c_{\frac{n-2}{2}}\int_{-1}^1 g(t)\frac{C_k^{\frac{n-2}{2}}(t)}{C_k^{\frac{n-2}{2}}(1)}(1-t^2)^{\frac{n-3}{2}}dt,
	\]
	with $C^\alpha_k$ the Gegenbauer polynomials defined in \eqref{eq:gegenbauer} and $c_\lambda$ is the normalization constant 
	\[c_\lambda = \left(\int_{-1}^1(1-t^2)^{\lambda - \frac{1}{2}}dt\right)^{-1} = \frac{\Gamma(\lambda+1)}{\sqrt{\pi}\Gamma(\lambda+\frac12)}.
	\]
\end{definition}
Spherical harmonics decomposition allows to represent any admissible $g$ as a linear combination of Gegenbauer polynomials.
\begin{lemma}[\cite{dai2013approximation}] 
\label{lemma:gegegnbauer-decomposition} 
Let $g: [-1, 1] \to \bbR$ satisfy the integrability condition~\eqref{eq:spherical:intergrability}, then $g$ has the following representation in terms of the Gegenbauer polynomials:
\[
g = \sum_{k} \hat g_k \frac{2k +n-2}{n-2} C_k^{\frac{n-2}{2}},
\]
where $(\hat g_k)_{k\in\bbN}$ is the spherical harmonics decomposition of $g$, $C_k^{\frac{n-2}{2}}$ are the Gegenbauer polynomials and the equality holds in $L_2(\bbS^{n-1})$ sense. 
\end{lemma}
\begin{remark}
	Note that the Definition \ref{def:spherical-decomposition} is an explicit formulation in terms of the Gegenbauer polynomials of the (simplified)
	Definition \ref{def:sph-decomposition-intro} used in the Introduction. The coefficients $\hat W_k$ are properly scaled projections onto the spherical harmonics basis functions $Y_{k,0}$. See section \ref{sec:harmonics} for the explicit expressions of the spherical harmonics $Y_{k,0}$.
\end{remark}

The following theorem shows that the eigenspaces of the Laplace-Beltrami operator $\calH_p$ are invariant under the convolution with any admissible $g$. 

\begin{theorem}[Invariance of $\calH_p$, {\cite[Theorem 2.1.3]{dai2013approximation}}]
    Let $f, g$ be as in Definition \ref{def:conv-sn}, then 
\[
\proj_p (g * f) = \omega_n \hat g_p \proj_p f,
\]
 where $\hat g_p$ denotes the spherical harmonics decomposition of $g$ as in Definition \ref{def:spherical-decomposition}.
\end{theorem}
As a result of the above theorem we can formulate an alternative of the convolution theorem in $\bbS^{n-1}$. 
\begin{corollary}[Convolution theorem on $\bbS^{n-1}$]
\label{cor:convolution-theorem}
Let $f, g$ be as in Definition \ref{def:conv-sn}, then for any $p\in \bbN, \ k \in \bbK_p$ the $(p, k)$-th spherical harmonics coefficient of the convolution $f* g$ satisfy
\[
\left<g*f, Y_{p, k}\right>_{L_2(\bbS^{n-1})} =  \omega_n\hat g_p\left<f, Y_{p, k}\right>_{L_2(\bbS^{n-1})}.
\]
\end{corollary}

In a given coordinate system $\theta_1, ... \theta_{n-1}$ we define the space of rotationally symmetric square-integrble functions as
\begin{definition}[Space $L_2^s(\bbS^{n-1})$]\label{def:L2s}
   Let $\theta_1, ... \theta_{n-1}$ be a coordinate system on $\bbS^{n-1}$. The space of $L_2^s(\bbS^{n-1})$ is the closed span of the spherical harmonics $Y_{l, 0}$ as defined in \eqref{eq:harmonics}:
   \[
   L_2^s(\bbS^{n-1}) = \overline{\operatorname{span}}^{L_2}\{Y_{l, 0},l\in\bbN\}.
   \]
\end{definition}
\begin{remark}
    By definition of spherical harmonics it is easy to see that any $f\in L_2^s(\bbS^{n-1})$ is a function of $\theta_{n-1}$ only and any harmonic $Y_{l, p}$ with $p\neq 0$ is orthogonal to $f$.
\end{remark}

\subsection{Admissible interaction kernels on a sphere}\label{ssec:InteractionSphere}

We endow the unit sphere $\bbS^{n-1}$ with the natural topology and assume that the interaction kernel satisfies the following assumption.
\begin{assumption}[Rotational symmetry] 
	\label{assum:sym-kernel}
	The interaction kernel $W: \bbS^{n-1}\times \bbS^{n-1} \to \bbR$ takes the form $W(x, y) = W(\left<x, y\right>)$, with $W:[-1,1]\to \R$ by abuse of notation and $\left<\cdot, \cdot\right>$ is the standard Euclidean product on $\bbR^n$.
\end{assumption}
In this case the measure convolution reduces to the spherical convolution defined in Section~\ref{sec:spherical-conv}. We discuss the spherical harmonics basis as well as properties of the convolution with rotationally symmetric function in Section \ref{sec:harmonics}. Properties of the spherical convolution are essential for the spectral analysis of the Gibbs map, which is the main component of the bifurcation analysis. For a kernel satisfying Assumption \ref{assum:sym-kernel} we define it's spherical harmonics decomposition following Definition~\ref{def:spherical-decomposition}.

For rotationally symmetric kernels spherical harmonics decomposition is an analogue of the Fourier decomposition in Euclidean case. In particular, set of all spherical harmonics is an orthonormal basis of $L_2(\bbS^{n-1})$ and a 'convolution' with rotationally symmetric kernel $W$ can be expresses solely in terms of the spherical harmonics decomposition of $W$. For any $u \in L_2(\bbS^{n-1})$ such that $u = \sum_{l, K} \hat u_{l, K} Y_{l, K}$ the convolution takes thanks to Corollary~\ref{cor:convolution-theorem} the form
\begin{equation}
	\label{eq:intro-conv}
	(W*u)(x) = \int_{\bbS^{n-1}} W(\left<x,y\right>) u(y)d\sigma(y) =\omega_n\sum_{l\in\bbN, K\in \bbK_l} \hat W_l \hat u_{l, K} Y_{l, K}(x),
\end{equation}
where $Y_{l, K}$ are the spherical harmonics from Section~\ref{sec:harmonics}. 
\begin{definition}[Stable kernels]\label{def:stable-kernel}
	A symmetric kernel $W$ satisfying the following condition
	\[
	\calI(u) = \frac12\int_{\bbS^{n-1}\times \bbS^{n-1}} W(\left<x,y\right>) u(x)u(y)d\sigma(x)d\sigma(y) \geq 0 \quad \forall u \in L_2(\bbS^{n-1})
	\]
	is called stable and the space of all stable kernels is denoted by $\calW_s$.
\end{definition}
\begin{lemma}[Spherical decomposition of a stable kernel]
	Let $W \in \calW_s$ be a stable kernel, then $\forall k \in \bbN$ it holds that $\hat W_k \geq 0$, where $(\hat{W}_n)_{n\in\bbN}$ is the spherical harmonics decomposition of $W$.
\end{lemma}
\begin{proof}
	Using the relation \eqref{eq:intro-conv} we obtain for any $u \in L_2(\bbS^{n-1})$:
	\begin{align*}
		\MoveEqLeft\calI(u) = \frac12\int_{\bbS^{n-1}\times \bbS^{n-1}} W(\left<x,y\right>) u(x)u(y)d\sigma(x)d\sigma(y) \\
		&= \frac{1}{2\bar\rho}\int_{\bbS^{n-1}}d\sigma(x) \sum_{l, K}\hat W_l \hat u_{l, K} Y_{l, K}(x) \sum_{l, K} \hat u_{l, K} Y_{l, K}(x)  
		=\frac{1}{2\bar\rho} \sum_{l, K} \hat W_l \hat u_{l, K}^2,
	\end{align*}
	where we used the orthonormality of spherical harmonics basis. Assume that there exists $l \in\bbN$ such that $\hat W_l < 0$, by taking $u^* = Y_{l, 0}$ we obtain $\calI(u^*) = \frac{1}{2\bar\rho}\hat W_l <0$, contradiction.
\end{proof}
\begin{remark}
	Note that the concept of a stable kernel on a sphere is equivalent to the characterization of $H$-stability of the interaction kernel on a torus used in \cite{carrillo2020long}. 
\end{remark}
As it was established in \cite{GatesPenrose1970} for the Euclidean setting, phase transition is only possible if the interaction kernel is $H$-unstable. Similarly, in this work for the bifurcation analysis and the characterization of the phase transitions we require the interaction kernel to be unstable, and in particular to satisfy the following assumption.
\begin{assumption}
	\label{assum:non-stable-kernel}
	The interaction kernel $W: \bbS^{n-1}\times \bbS^{n-1} \to \bbR$ satisfies Assumption~\ref{assum:sym-kernel} and has at least one negative component of the spherical harmonics decomposition, i.e. $W \not\in \calW_s$.
\end{assumption}
\begin{remark}[Convexity of the interaction energy for stable kernels]
\label{remark:stable-kernels}
As noted in \cite[Corollary 2.7]{ChayesPanferov2010}, the interaction energy $\calI$ with $W\in \calW_s$ is a positive-semidefinite bilinear map on $L_2(\bbS^{n-1})$, so the uniform state $\bar\rho$ the unique minimizer of the free energy $\calF_\gamma$ for arbitrary~$\gamma \in \bbR$.
\end{remark}

\section{Sphere: Bifurcation analysis}\label{sec:bifurcations}
As shown before, for large enough noise the uniform measure is a unique minimizer of the aggregation-diffusion potential. We would like to know what happens to the set of minimizers as we decrease the noise. As one can see, for a fixed interaction kernel $W$ for certain values of $\gamma$ the convexity parameter $\lambda$ becomes negative, which may result in non-uniqueness of minimizers. Under certain conditions, the set of minimizers can start 'branching' around certain state $\rho$. This phenomenon is called bifurcation. In this section we give the necessary background information on the bifurcation theory in Hilbert spaces and apply it to the Gibbs map defined in \eqref{eq:gibbs-map}.
\subsection{Background on general theory}
We study behaviour of the set of solutions of the parametrized family of fixed point equations
\begin{equation}
\label{eq:fixed-point}
F(x, \lambda)=0,
\end{equation}
 for some $F: X\times \bbR \to X$, where $X$ is a Banach space and $\lambda$ is the parameter. The parameter point $\lambda_0$ combined with $x_0$, the solution of \eqref{eq:fixed-point} at which the set of minimizers $X_\lambda = \{x: F(x, \lambda_0) =0\}$ starts branching is called a bifurcation point. In particular by the branching we understand existence of solutions of \eqref{eq:fixed-point} different from $(x_0, \lambda_0)$ in any small proximity of $(x_0, \lambda_0)$.

 \begin{definition}[Bifurcation point]
Consider a Banach space $X$ with $U \subset X$ an open neighbourhood of $0$, and a nonlinear $C^2$ map, $F: U \times V \rightarrow X$, then a point $(x_0, \lambda_0) \in U \times V$ is a \emph{bifurcation point} of equation $F(x, \lambda) = 0$ if 
\[
(x_0, \lambda_0) \in \text{cl}\left\{ (x, \lambda): F(x, \lambda) = 0, \ x \neq x_0, \ \lambda \neq \lambda_0\right\}.
\]    
\end{definition}

 If the set of $(x, \lambda)$ form a curve in $X \times \bbR$, it's called a \emph{bifurcation curve}.

 \begin{definition}[Bifurcation curve]
Let $(x_0, \lambda_0)$ be a bifurcation point of $F$. If there exists $\varepsilon >0$ and a continuous map $S: (-\varepsilon, \varepsilon) \to X \times \bbR$ such that $S(0) = (x_0, \lambda_0)$ and for all $t \in (\varepsilon, \varepsilon)$:
\[
F(S(t)) = 0,
\]
then $(x(t), \lambda(t)) = S(t)$ is a bifurcation curve.
\end{definition}
The bifurcation theory often relies on the finite-dimensional inverse function theorem and thus sets additional assumptions of the functional $F$. Let $F$ take the form
\begin{equation}
\label{eq:k+g}
F(x, \lambda) =x - \lambda Kx + G(x, \lambda), 
\end{equation}
for some linear functional $K$, and a non-linear part $G$ with a bounded growth. The exact conditions will be made explicit later. To be able to reduce the problem to a finite-dimensional setting, we set an assumption on the linear part $K$ to be a Fredholm operator.
\begin{definition}
    Let $X, Y$ be Banach spaces and $F$ be a linear bounded operator $F: X \to Y$. $F$ is a \emph{Fredholm} operator of index $k$ if $\dim \ker F = d_1 <\infty$, $\dim (Y\backslash im(F)) = d_2 <\infty$ and $k = d_1 - d_2$,
    where $\operatorname{im}(F) =\{F(x): x\in X\}$.
\end{definition}
We will make use of the following properties of the Fredholm operators:
\begin{lemma}[Lemma 4.3.3 \cite{davies2007linear}]
\label{lemma:Fredholm-compact}
Let $X$ be a Banach space and let $A: X\to X$ be a compact operator, then for any $\lambda \neq 0$ operator $\lambda I -  A$ is Fredholm.
\end{lemma}
As it is clear from the above lemma, operator $1/\lambda(\lambda I -  A)  = I - 1/\lambda A$ is also a Fredholm operator of the same index.
\begin{lemma}[Theorem 4.3.11  \cite{davies2007linear}]
\label{lemma:fredholm-continuity}
If $A: X\to Y $ a Fredholm operator then there exists $\varepsilon > 0$
such that every bounded operator $X$ such that $\|A- X\| < \varepsilon$ is also Fredholm
with $\operatorname{index}(X) = \operatorname{index}(A)$.
\end{lemma}
\begin{corollary}
\label{cor:fredholm-family}
    Let $(A_t)_{t\in [0,1]}$ be a norm-continuous family of Fredholm operators $A_t: X\to Y$. Let $A_0$ be a Fredholm operator of index $k$, then $\operatorname{index}(A_t) = k$ for all $t\in [0,1]$. 
\end{corollary}
\begin{proof}
    Consider the map $I: [0,1] \to \bbN_0$ defined as $I(t):= \operatorname{index}(A_t)$. If there exists $t$ such that $I(t) \neq k$, there exist at least one point of discontinuity of $I$. Take such a discontinuity point $t_0$, then by Lemma \ref{lemma:fredholm-continuity} there exists $\varepsilon >0$ such that map $I$ is constant on $B(A_{t_0}, \varepsilon)$. As the family $A_t$ is norm-continuous, there exists $\delta > 0$ such that for all $t \in (t_0 -\delta, t_0+\delta): \ \|A_t - A_{t_0}\| < \varepsilon$, so the map $I$ is constant on $(t_0 -\delta, t_0+\delta)$, contradiction. 
\end{proof}
We are now ready to formulate the sufficient condition of a bifurcation curve around $(x_0, \lambda_0)$, which we apply afterwards. 
\begin{theorem}[Existence of a bifurcation curve {\cite[Theorem 28.3]{deimling2013nonlinear}}]
\label{th:bifurcation-curve}
Let $X$ be a real Banach space and $F: U \times V \to X $ be an operator of form \eqref{eq:k+g}, where $U \times V \subseteq X\times \bbR$ is a neighborhood of $(0, \lambda_0)$. Let $K$ be a bounded linear operator and $G: U \times V \to X$ is such that $G_x, G_\lambda, G_{x,\lambda}$ are continuous on $U \times V$. In addition suppose that 
    \begin{itemize}
        \item $I - \lambda_0 K$ is a Fredholm operator of index zero and $1/\lambda_0$ is a simple characteristic value of $K$ with the corresponding eigenvector $v_0$.
        \item $G(x, \lambda) =o(\|x\|)$ as $x\to 0$ uniformly in $\lambda$ on $U \times V$,
    \end{itemize}
    then $(0, \lambda_0)$ is a bifurcation point of $F(x, \lambda)$. In addition, there exists $\delta > 0$ and a bifurcation curve $S: (-\delta, \delta) \to X \times \bbR$, $S(t) = (x(t), \lambda(t))$ such that $x(t), \lambda(t)$ take the form
    \[
    x(t) = tv_0 + tz(t), \qquad \lambda(t) = \lambda_0 + \mu(t)
    \]
    where $\left<z(t), v_0\right> =0 $ for all $|t| <\delta$ and $z, \mu$ are continuous functions satisfying $z(0) =0, \ \mu(0) = 0$. %
\end{theorem}

\subsection{Main result}
In the previous section we've established equivalence between the set of solutions of the stationary McKean-Vlasov equation and the set of zeros of the Gibbs map. In this section we study bifurcations of the Gibbs map on an $n$-dimensional unit sphere around the uniform measure $\bar \rho$. First of all note that because of the symmetry of the problem the uniform measure is a solution of $F(x, \gamma)=0$ for any $\gamma\in\bbR_+$.
\begin{lemma}
\label{lemma:barrho}
Let $W$ satisfy Assumption \ref{assum:sym-kernel}, then for any $\gamma\in\bbR_+$ the pair $(\bar\rho, \gamma)$ is a zero of the Gibbs map, i.e.: 
\[
\bar\rho - \frac{1}{Z(\gamma, \bar \rho)}e^{-\gamma W*\bar \rho} = 0.
\] 
\end{lemma}
\begin{proof}
    Note that by definition for any $x, y \in \bbS^{n-1}$ it holds that $\bar \rho(x) = \bar \rho(y)$, implying that for any interaction kernel of form $W(x, y) = W(\left<x, y\right>)$ the following equality is satisfied $e^{-\gamma (W*\bar \rho)(x)} = e^{-\gamma (W*\bar \rho)(y)}$. Finally note that  
    \[
    \int_{\bbS^{n-1}}\frac{1}{Z(\gamma, \bar \rho)}e^{-\gamma W*\bar \rho}d\sigma = \frac{e^{-\gamma (W*\bar \rho)(x_0)}\sigma(\bbS^{n-1})}{Z(\gamma, \bar \rho)} =1 = \bar{\rho}(x_0)\sigma(\bbS^{n-1}).
    \]
    The statement of the lemma immediately follows.
\end{proof}
In fact, due to the symmetric structure of the sphere, all the minimizers of the free energy functional can be shown to admit certain symmetries. Similar result has been established in~\cite{BaernsteinTaylor1976} using symmetrization argument, see also~\cite[Chapter 4]{hayman1994multivalent}.

Define $\hat F: L_2(\bbS^{n-1})\times \bbR_+ \to L_2(\bbS^{n-1})$ as a shifted version of $F$:
\begin{equation}
   \label{eq:hatF} 
\hat F(u, \gamma) = F(\bar \rho + u, \gamma) = \bar \rho + u - \frac{1}{Z(\gamma, u)} e^{-\gamma W* (\bar \rho + u)},
\end{equation}
where $Z( \gamma, u) = \int_{\bbS^{n-1}} e^{-\gamma W* (\bar \rho + u)}d\sigma $. Then the following trivial representation of $\hat F$ holds
\[
\hat F(u, \gamma) = D_u \hat F(0, \gamma) [u] + \left(\hat F(u, \gamma) - D_u \hat F(0, \gamma) [u]\right),
\]
where $D_u \hat F(\bar\rho, \gamma)$ is the Frechet derivative of $\hat F$ and $\left(\hat F(u, \gamma) -  D_u \hat F(0, \gamma) [u]\right)$ is the non-linear term. Note that for any bifurcation curve $\hat S(t) = (\hat x(t), \hat \gamma(t))$ of the shifted functional $\hat F$ there exists corresponding bifurcation curve of $F$ which takes the form $S(t) = (\bar\rho + \hat x(t), \hat \gamma(t))$, and thus it is sufficient to establish bifurcation points of $\hat F$.

Recall that for a rotationally symmetric interaction kernel according to Corollary \ref{cor:convolution-theorem} for any $u\in L_2(\bbS^{n-1})$ it holds that
\[
\left<W*u, Y_{k, p}\right>_{L_2(\bbS^{n-1})} =  \omega_n \hat W_k\left<u, Y_{k, p}\right>_{L_2(\bbS^{n-1})},
\]
where $(\hat W_k)_{k\in \bbN}$ is the spherical harmonics decomposition of $W$. Using this property we are able to prove the following theorem.
\begin{theorem}
\label{th:bifurcations}
Let $W \in C_b(\bbS^{n-1} \times \bbS^{n-1})$ satisfy Assumption \ref{assum:non-stable-kernel}. If $\hat W_k < 0$ is a unique component of the spherical harmonics decomposition $(\hat W_k)_{k\in \bbN}$, then $(0, \gamma_k)$ with $\gamma_k =- \frac{1}{\hat W_k}$ is a bifurcation point of $\hat F$. Moreover there exists at least one bifurcation curve $\hat S(t) = (\hat x(t), \gamma(t))$ of form
\[
\hat x(t) = tY_{k, 0} + tz(t), \qquad \gamma(t) = \gamma_k + \mu(t),
\]
where $z, \mu$ are continuous functions satisfying $z(0)=0, \ \mu(0)= 0$ and $\hat x(t) \in L_2^s(\bbS^{n-1})$ for all $t \in (-\delta, \delta)$, with $L^s_2(\bbS^{n-1})$ from Definition~\ref{def:L2s}.
\end{theorem}
In order to prove the above theorem we require the following technical lemmas, for which we provide the proofs in the Appendix \ref{appendix:bifurcations}. 
\begin{lemma}
\label{lemma:G}
    Let $W \in C_b(\bbS^{n-1} \times \bbS^{n-1}) $ satisfy Assumption \ref{assum:non-stable-kernel}. Let $G:L_2(\bbS^{n-1}) \times \bbR_+$ be the non-linear operator as defined in \eqref{eq:G},
    then for any $\gamma\in\bbR_+$ there exists a neighborhood of $(0, \gamma)$: $U \times V \subseteq L_2(\bbS^{n-1})\times \bbR$, such that $G_u, G_\gamma, G_{u\gamma}$ are continuous on $U \times V$. Moreover, $G(w, \gamma) =o(\|w\|_{L_2})$ as $w\to 0$ uniformly in $\gamma$ on $U \times V$.
\end{lemma}

\begin{lemma}
\label{lemma:compactness}
    Let $(\calM, \dist)$ be a compact smooth Riemannian manifold, let $\sigma$ be the uniform measure on $\calM$ and let $W\in C_b(\calM\times\calM, \bbR)$, then the operators $A, B: L_2(\calM) \to L_2(\calM)$
    \[
    Au:= W * u \quad \text{ and } \quad Bu:= \sigma \int_\calM W * ud\sigma
    \]
    are compact.
\end{lemma}
We are now ready to prove the main theorem of this section.

\begin{proof}[Proof of Theorem~\ref{th:bifurcations}]
The structure of the proof is the following, we first give an explicit form of both linear and non-linear parts of $\hat F$. Then we show that the space of symmetric functions $L_2^s$ is invariant under $\hat F$ and thus we can study the restriction of $\hat F$ to $L_2^s$. Finally, we show that the restricted operator satisfies the assumptions of Theorem \ref{th:bifurcation-curve}. 

\smallskip\noindent
\textbf{Linear and non-linear parts of $\hat F$:} By definition of Frechet derivative we have
\begin{align*}
\hat{F}(u+w, \gamma) - \hat{F}(u, \gamma) = D_u \hat F(u, \gamma) [w] + o(\|w\|).
\end{align*}
Expanding the left hand side we get
\begin{align*}
   \MoveEqLeft \hat{F}(u+w, \gamma) - \hat{F}(u, \gamma) - w= \frac{1}{Z(\gamma, u)} e^{-\gamma W* (\bar \rho + u)} - \frac{1}{Z(\gamma, u + w)} e^{-\gamma W* (\bar \rho + u + w)} \\
   &=\frac{1}{Z(\gamma, u)} e^{-\gamma W* (\bar \rho + u)}\left(1 - \frac{Z(\gamma, u)}{Z(\gamma, u + w)}e^{-\gamma W* w}\right) \\
   &=\frac{1}{Z(\gamma, u)}  e^{-\gamma W* (\bar \rho + u)}\cdot \gamma W* w - \frac{e^{-\gamma W* (\bar \rho + u)}}{Z(\gamma, u)^2} \left(Z(\gamma, u + w) - Z(\gamma, u)\right) + o(\|w\|) \\
   &=\frac{1}{Z(\gamma, u)}  e^{-\gamma W* (\bar \rho + u)}\cdot \gamma W* w - \frac{ \gamma e^{-\gamma W* (\bar \rho + u)}}{Z(\gamma, u)^2}\int_{\bbS^{n-1}} e^{-\gamma W* (\bar \rho + u)} \cdot W* w d\sigma + o(\|w\|)
\end{align*}
Then the first Frechet derivative $D_u \hat F$ for variation $w \in L_2(\bbS^n)$ satisfies:
\begin{equation}
\label{eq:DwF}
D_u \hat F(u, \gamma) [w] = w + \frac{1}{Z(\gamma, u)}  e^{-\gamma W* (\bar \rho + u)} \cdot \gamma W* w - \frac{ \gamma e^{-\gamma W* (\bar \rho + u)}}{Z(\gamma, u)^2}\int_{\bbS^{n-1}} e^{-\gamma W* (\bar \rho + u)} \cdot W* w d\sigma,
\end{equation}
in particular for $u = 0$ we have
\[
D_u \hat F(0, \gamma) [w] = w + \gamma \bar \rho \cdot  W* w - \gamma \bar\rho^2  \int_{\bbS^{n-1}}  W* w d\sigma,
\]
and can define the linear and non-linear parts as
\begin{align}
Kw&:= -\bar \rho \cdot  W* w + \bar\rho^2  \int_{\bbS^{n-1}}  W* w d\sigma, \label{eq:K} \\
G(w, \gamma) &:=\bar \rho - \frac{1}{Z(\gamma, w)} e^{-\gamma W* (\bar \rho + w)} - \gamma \bar \rho \cdot  W* w + \gamma \bar\rho^2  \int_{\bbS^{n-1}}  W* w dx. \label{eq:G}
\end{align}

\smallskip\noindent
\textbf{Invariance of $L_2^s$:}
Recall that the subset of symmetric spherical harmonics $(Y_{l, 0})_{l\in\bbN}$ is an orthonormal basis for $L_2^s(\bbS^{n-1})$. In addition, we prove $L_2^s(\bbS^{n-1})$ is invariant under $K$. For any $w \in L_2^s(\bbS^{n-1})$ we obtain:
\begin{align*}
\MoveEqLeft \norm[\bigg]{ Kw - \sum_{l\in\bbN}\left<Kw, Y_{l,0}\right>Y_{l,0}}_{L_2}  = \norm[\bigg]{Kw - \sum_{l\in\bbN}\proj_l Kw }_{L_2} = 0,
\end{align*}
where we used that for any $m \in \bbN, \ p\neq 0$:
\begin{align*}
\MoveEqLeft \left<Kw, Y_{m,p}\right> = \left<\proj_m Kw, Y_{m,p}\right> = \biggl<\proj_m K\sum_{l\in\bbN}\left<w, Y_{m,0}\right>Y_{l, 0}, Y_{m,p}\biggr> \\
&= -\int \bar \rho \cdot \proj_m \biggl( W* \sum_{l\in\bbN}\left<w, Y_{l,0}\right>Y_{l, 0} \biggr)Y_{m, p}d\sigma \\
&\qquad +  \int_{\bbS^{n-1}}\bar\rho^2 \cdot \proj_m\biggl(\int_{\bbS^{n-1}}  W* wd\sigma \biggr)Y_{m, p} d\sigma \\
&=-\int \bar \rho \cdot\proj_m \biggl(\sum_{l\in\bbN}\hat W_l \left<w, Y_{l,0}\right>Y_{l, 0}\biggr) Y_{m, p}d\sigma +0 \\
&= -\int \bar \rho \cdot\hat W_m \left<w, Y_{m,0}\right>Y_{m, 0} Y_{m, p}d\sigma = 0.
\end{align*}

\smallskip\noindent
\textbf{Properties of $K$:} We will now show that $I -\gamma K$ is a Fredholm operator of index zero for any $\gamma \in \bbR_+$ as an operator on the space $L_2^s(\bbS^{n-1})$. 
By Lemma \ref{lemma:compactness} operator $K$ is compact on $L_2(\bbS^{n-1})$. Since $L_2^s(\bbS^{n-1})$ is invariant under $K$ and $L_2^s(\bbS^{n-1})$ is a closed subspace of $L_2(\bbS^{n-1})$, it is also a compact operator on the smaller space $K:L_2^s(\bbS^{n-1}) \to L_2^s(\bbS^{n-1})$. By Lemma \ref{lemma:Fredholm-compact} $I -\gamma K$ is Fredholm. Moreover the family $(I-\gamma K)_{\gamma \in \bbR^+}$ is norm-continuous and as a result $\operatorname{index}(I-\gamma K) = \operatorname{index}(I) = 0$ by Corollary \ref{cor:fredholm-family}. We then calculate for $Y_{0,0}$
\[
KY_{0, 0} = -\frac{1}{\omega_n} W*Y_{0, 0} + \frac{1}{\omega_n^2} \int_{\bbS^{n-1}}  W* Y_{0, 0} d\sigma = -\hat W_0 Y_{0,0} + \frac{1}{\omega_n} Y_{0,0}\int_{\bbS^{n-1}}\hat W_0 Y_{0,0}d\sigma=0.
\]
and for the other harmonics $Y_{l, 0}, l\in\bbN$ 
\[
KY_{l, 0} = -\frac{1}{\omega_n} W*Y_{l, 0} + \frac{1}{\omega_n^2} Y_{0,0} \int_{\bbS^{n-1}}  W*Y_{l, 0} d\sigma = -\hat W_0 Y_{l,0}+ 0= -\hat W_0 Y_{l,0}.
\]
So in the orthonormal basis of spherical harmonics of order zero $K$ is a diagonal operator
\[
KY_{l, 0} = \begin{cases}
    0, \qquad &l=0,
    \\
   -\hat W_lY_{l, 0}, \qquad &l\in \bbN,
\end{cases}
\]
which under assumption of the theorem implies that $-1/\hat W_k$ is a simple characteristic value of $K$ and $Y_{k, 0}$ is the corresponding eigenvector. Thus we can apply Theorem \ref{th:bifurcation-curve}, where necessary regularity of the nonlinear term $G$ follows from Lemma \ref{lemma:G}.
\end{proof}

\section{Sphere: Phase transitions}
As shown in Lemma \ref{lemma:barrho} the uniform measure is always a critical point of the free energy functional. Moreover, since sphere has positive curvature, if $\gamma$ is small enough then $\calF$ is strictly geodesically convex and thus uniform measure is a unique global minimizer. As shown is Section~\ref{sec:bifurcations} if the kernel $W$ is unstable, the set of stationary points of $\calF$ behaves non-trivially for larger values of $\gamma$, so $\bar \rho$ is not guaranteed to be the global minimizer anymore. The value $\gamma$ at which the uniform measure stops being the global minimizer of the free energy functional is called the \emph{transition point}. In this section we give a necessary condition for an existence of a transition point $\gamma_c$ and characterize the minimizers of $\calF_\gamma$ around it.

\subsection{Background}

\begin{definition}[Transition point]\label{def:transition-point}
    A parameter value $\gamma_c \in\bbR_+$ is a \emph{transition point} of the free energy functional \eqref{eq:free-energy} if the following holds
    \begin{enumerate}
        \item for all $\gamma \in (0,\gamma_c)$ the uniform measure $\bar\rho$ is a unique (global) minimizer of $\calF_\gamma$,
        \item for $\gamma =\gamma_c$ the uniform measure is a minimizer of $\calF_\gamma$,
        \item for any $\gamma>\gamma_c$ there exist at least one probability measure $\rho_\gamma \in \calP_{ac}^+(\bbS^{n-1})$ different from $\bar\rho$, minimizing $\calF_\gamma$.
    \end{enumerate}
\end{definition}
To give a characterization of the phase transition in the spherical case we introduce the result analogous to the analysis of phase transitions of the McKean-Vlasov equation on a torus from \cite{ChayesPanferov2010}. Since the proof does not rely on the structure of the domain, all the arguments from \cite{ChayesPanferov2010} directly translate to the case of a (high-dimensional) sphere. 
\begin{lemma}
\label{lemma:properties}
    Let $\calF_\gamma$ be the free energy functional~\eqref{eq:free-energy} with fixed $\gamma>0$ and let $W \in \mathcal{W}_s^c$ be an unstable interaction potential (see Definition~\ref{def:stable-kernel}). %
    Suppose that for some $\gamma_T <\infty$ there exists a minimizer $\rho_T \in \calP_{ac}(\bbS^{n-1})$ not equal to the uniform measure $\bar\rho $ such that: 
\[
\calF_{\gamma_T}(\rho_T) \leq \calF_{\gamma_T}(\bar \rho)
\]
then, for all $\gamma > \gamma_T$ the uniform measure no longer minimises the free energy.

\end{lemma}
Existence of the phase transition point depends on the properties of the interaction kernel~$W$. Namely, instability of the interaction kernel is a necessary and sufficient condition of a phase transition as well. This result dates back to \cite{GatesPenrose1970}, where the analog of the following proposition is proved in Euclidean setting, for completeness we provide a short proof below.
\begin{proposition}[Existence of phase transition point]
\label{prop:existence-pt}
    Consider the free energy functional \eqref{eq:free-energy} with rotationally symmetric interaction kernel~$W$. The system exhibits phase transition if and only if $W\in\calW_s^c$. 
\end{proposition}
\begin{proof}
    Necessity follows immediately from the Remark \ref{remark:stable-kernels} since for $W\in\calW_s^c$ the uniform state is the unique minimizer of the free energy functional.
    
    We now prove sufficiency. Since $W\in\calW_s^c$, there exists $k\in \bbN$ such that the corresponding coefficient of the spherical harmonics decomposition $\hat W_k <0$. Choose such a $k$ and consider the competitor $\rho_k = \bar\rho + \varepsilon Y_{k, 0}$ for some $\varepsilon >0$ but small enough to ensure positivity of the density $\rho_k > 0$. Calculating the interaction energy for $\rho_k$ we obtain
    \[
    \calI(\rho_k) = \frac12\int W(x, y)(\bar\rho + \varepsilon Y_{k, 0}(x)) (\bar\rho + \varepsilon Y_{k, 0}(y))d\sigma(x)d\sigma(y) = \calI(\bar\rho) + \frac{\varepsilon^2}{2\bar\rho}\hat W_k < \calI(\bar\rho).
    \]
    Since the spherical harmonics are continuous functions, $\rho_k$ is continuous and thus is upper bounded $\rho_k < \varrho \in \bbR_+$. As a result we can estimate the entropy at $\rho_k$ as 
    \[
    \calE(\rho_k) \leq \int \rho_k\log \rho_k d\sigma \leq \int \rho_k \log \varrho d\sigma= \log \varrho.
    \]
    Combining the estimates we conclude that the difference $\calF_\gamma(\rho_k) -\calF_\gamma(\bar\rho)$ has the following upper-bound:
    \[
    \calF_\gamma(\rho_k) -\calF_\gamma(\bar\rho) = \gamma^{-1}(\calE(\rho_k)- \calE(\bar \rho)) +\calI(\rho_k)- \calI(\bar \rho) \leq \frac{\varepsilon^2}{2\bar\rho}\hat W_k + \gamma^{-1}(\log \varrho - \log \bar \rho).
    \]
    Choosing $\tilde\gamma > -\frac{\bar\rho(\log \varrho - \log \bar \rho)}{\varepsilon^2\hat W_k}$ we get $\calF_{\tilde\gamma}(\rho_k) -\calF_{\tilde\gamma}(\bar\rho) <0$ and thus $\bar \rho$ is not a minimizer of the free energy functional $\calF_{\tilde\gamma}$. We aim to show that there exists a parameter value $\gamma_c $ satisfying definition \ref{def:transition-point}. Since $\bar\rho$ is always a critical point of $\calF$, from the convexity properties of $\calF$ we know that for small enough $\gamma$ the uniform state $\bar \rho$ is the unique minimizer of the free energy. Since for any $\gamma\in\bbR$ there exists a minimizer of $\calF_\gamma$, by Lemma \ref{lemma:properties}, we conclude that there exist a point $\gamma_c \leq \tilde\gamma$ such that for every $\gamma > \gamma_c$ there exists a minimizer of the free energy functional $\rho_\gamma \neq \bar\rho$. Let 
    \[\gamma^*_c = \sup\{\gamma: \bar\rho \text{ is a minimizer of } \calF_\gamma\},
    \]
    then for $\gamma < \gamma^*_c$ the uniform state is the unique minimizer of the free energy.
    We will show that $\gamma^*_c$ is a phase transition point. If $\bar\rho$ is a minimizer of $\calF_{\gamma^*_c}$, then by Theorem \ref{th:minimizers} for every $\gamma > \gamma^*_c$ there exist a minimizer of the free energy different from $\bar\rho$, so $\gamma^*_c$ is a transition point. We thus only need to show that $\bar\rho$ is a minimizer of $\calF_{\gamma^*_c}$. 
    
    Assume that $\bar\rho$ is not a minimizer of $\calF_{\gamma^*_c}$, then by
    Theorem \ref{th:minimizers} there exists $\rho_{\gamma^*_c} \in L_2(\bbS^{n-1}) \neq \bar\rho$ minimizing $\calF_{\gamma^*_c}$. For a fixed $\rho \in L_2(\bbS^{n-1})$ the map $\varGamma_\rho: \bbR_+ \to \bbR$
    \[
    \varGamma_\rho(\gamma ) = \calF_\gamma(\rho )
    \]
    is continuous, so the map
    \[
    f(\gamma) :=\varGamma_{\bar\rho}(\gamma) - \varGamma_{\rho_{\gamma^*_c}}(\gamma)
    \]
    is also continuous on $\bbR_+$. Since for $\gamma < \gamma^*_c$ the uniform state is the unique minimizer, we have $f(\gamma) <0$ on $(0, \gamma^*_c)$. At the same time by our assumption $\calF_{\gamma^*_c}(\bar\rho) > \calF_{\gamma^*_c}(\rho_{\gamma^*_c})$, so  $f(\gamma^*_c) > 0$, which contradicts continuity, thus $\bar\rho$ is a minimizer of $\calF_{\gamma^*_c}$ and thus $\gamma^*_c$ is a point of phase transition of the corresponding free energy functional.  
\end{proof}

Phase transition points are called \emph{continuous} or \emph{discontinuous} depending on the structure of the set of minimizers around $\gamma_c$.
\begin{definition}[Continuous and discontinuous transition points]
    A transition point $\gamma_c$ is called a \emph{continuous} transition point of the free energy functional \eqref{eq:free-energy} if $\bar\rho$ is a unique minimizer for $\gamma = \gamma_c$ and for any family of minimizers $\{\rho_\gamma \in \calP_{ac}(\bbS^{n-1})| \gamma > \gamma_c\}$ it holds that
    \[
    \lim_{\gamma\downarrow \gamma_c}\|\rho_\gamma - \bar\rho\|_{L_1} = 0.
    \]
    \emph{Discontinuous} transition point is any transition points violating at least one of the above conditions.
\end{definition}
\begin{remark}
    Our definition of a continuous transition point is equivalent to the definition used in \cite{ChayesPanferov2010} and \cite{carrillo2020long} since norm is non-negative and thus $\limsup$ coincides with the limit.
\end{remark}

The following Lemma relates continuous transition points with the points of linear stability, namely the parameter value $\gamma_\#$ at which the linearized version of the free energy functional at $\rho=\bar\rho$ becomes unstable (admits a non-negative eigenvalue). We give a formal definition as well as the necessary linear stability results in the following subsection.
\begin{lemma}[\cite{ChayesPanferov2010}] 
\label{lemma:tp_lsp}Let $\calF_\gamma$ be the free energy functional with fixed $\gamma$ and the interaction potential $W \in C_b$. Assume that there exists at least one $k\in \bbN_0$ such that $\hat W_k <0$ and assume that $\calF$ exhibits a continuous transition point at some $\gamma_c < \infty$. Then the phase transition point coincides with the point of linear stability $\gamma_c = \gamma_\#$.
    
\end{lemma}
\subsection{Linear stability analysis}
We define the linear operator $\bar \calL: L_2(\bbS^{n-1}) \to  L_2(\bbS^{n-1})$ as a linearization of the non-linear PDE \eqref{eq:mckean-vlasov} at $\rho = \bar\rho$, namely:
\[
\bar \calL w := \Delta w +\gamma \bar\rho \Delta (W * w),
\]
for every $w\in L_2(\bbS^{n-1}) $. Using the properties of spherical convolution we conclude that
\[
\bar \calL Y_{l, K} = -l(n+l-2)(1 + \gamma\hat W_l)Y_{l, K}. 
\]
In other words basis of spherical harmonics of order $l$ forms the eigensubspace of $\bar L$ corresponding to the eigenvalue
\[
\lambda_l =  -l(n+l-2)(1 + \gamma\hat W_l).
\]
If there exists at least one negative element of the spherical harmonics decomposition of $W$, $\hat W_k <0$, there also exists a parameter value $\gamma_k = -\frac{1}{\hat W_k}$ where the corresponding eigenvalue $\lambda_k$ changes the sign. The minimum of such values 
\begin{equation}\label{eq:def:gamma_sharp}
\gamma_\# = -\frac{1}{\min_{k\in \bbN}\hat W_k},
\end{equation}
is called the \emph{point of linear stability}.
\begin{remark}[Uniform state is an unstable solution for $\gamma > \gamma_\#$]
    Mirroring the stability analysis of the linearized PDE, we remark that for $\gamma > \gamma_\#$ the uniform measure $\bar\rho$ is an unstable saddle point of the free energy functional $\calF_\gamma$ (namely $D^2_\rho \calF_\gamma(\bar\rho)$ is not a positively semidefinite bilinear operator). First recall that $\bar \rho$ is always a critical point by the symmetry argument. To show the latter let $l \in \argmin_{k\in \bbN}\hat W_k$ and consider the function $f_t(t) = \calF_\gamma(\bar\rho + tY_{l, 0})$. Differentiating $f(t)$ twice we obtain
    \[
    \frac{d^2}{dt^2}f(t)\Big|_{t=0} = D^2_{\rho} \calF_\gamma(\bar\rho)[Y_{l, 0}, Y_{l, 0}] =\frac{1}{\bar\rho}(1+ \gamma\hat W_l) <0,
    \]
    so $\bar\rho$ is an unstable critical point for any $\gamma > \gamma_\#$.
\end{remark}
\subsection{Main result}
\label{ssec:transition}
In this section we present a sufficient condition for the existence of a discontinuous phase transition point of McKean-Vlasov equation on a high-dimensional sphere. Let $\hat W_k$ be the minimal eigenvalue of the linearized operator discussed in the previous section. We introduce a criterion which relies on the properties of the eigenvectors corresponding to $\hat W_k$. We also show that criterion is trivially satisfied if $k = 2$ or $k=4$. 

In this Section we require the interaction potential to satisfy the (relaxed) resonance condition Assumption~\ref{assum:pt-almost} below. To make a parallel with the analysis of phase transitions on a torus \cite{carrillo2020long} we call it a \emph{resonance condition} even though on a sphere it may be satisfied even if the corresponding eigenspace is one-dimensional.

\begin{assumption}[Relaxed resonance condition]
\label{assum:pt-almost}
 Given an interaction potential $W$ satisfying Assumption \ref{assum:non-stable-kernel} and denote the set of $\delta$-resonating modes as
 \[
 K_{\sharp,\delta} = \set*{k \in \bbN: \hat W_k \leq -\frac{1-\delta}{\gamma_{\#}} }.
 \]
The kernel $W$ satisfies the relaxed resonance condition with bandwidth $\delta$ if the following holds:
\begin{itemize}
    \item there exists a linear combination of spherical harmonics $u = \sum_{l\in K_{\alpha, \delta}} c_l Y_{l, 0}$ such that
 \begin{equation}\label{eq:delta-resonant-U3}
 \int_{\bbS^{n-1}} u^3d\sigma = U_3 \neq 0,  \quad \text{and} \quad |u|_\infty \leq 1 \,;
\end{equation}
 \item  the bandwidth is small enough, namely $ \delta < \min\Bigl(\frac14, \frac{U_3^2}{49\sigma(\bbS^{n-1})^2}\Bigr)$ for some $u$ satisfying~\eqref{eq:delta-resonant-U3}.
 \end{itemize}
    
\end{assumption}

For convenience of the reader we also provide an hierarchy of stricter resonance conditions. We start by introducing the (strict) resonance condition, namely the special case of 
the Assumption \ref{assum:pt-almost} with $\delta= 0$.
\begin{assumption}[Resonance condition]
 \label{assum:pt-general}   
 Given an interaction potential $W$ satisfying Assumption \ref{assum:non-stable-kernel} and denote the set $K_\sharp = \{k \in \bbN: \hat W_k = -\gamma_{\#}^{-1} \}$ (note~\eqref{eq:def:gamma_sharp}). We say that $W$ satisfies the resonance condition if there exists a finite linear combination $u = \sum_{l\in K_{\alpha}} c_l Y_{l, 0}$ such that
 \[
 \int_{\bbS^{n-1}} u^3d\sigma = U_3 \neq 0, \quad %
 \]
 where $Y_{l, 0}$ are the corresponding spherical harmonics.
\end{assumption}
Note that the supremum norm of the competitor $u$ is uniformly bounded for any finite linear combination after a proper rescaling, which is not necessarily true for arbitrary $u$ due to the lack of the uniform (in $x$ and $k$) bound on the absolute value of the basis functions $Y_{k, 0}(x)$. That is why in the Assumption \ref{assum:pt-general} we restricted the set of possible competitors to the finite linear combinations. %
In particular, it may be enough to have a single eigenfunction corresponding to the minimal eigenvalue. In this case the Assumption~\ref{assum:pt-general} reduces to the Assumption \ref{assum:pt} . 
\begin{assumption}[Single component resonance]
 \label{assum:pt}   
 Given an interaction potential $W$ satisfying Assumption \ref{assum:non-stable-kernel} and let $k_{\min} = \argmin_{k\in \bbN}\hat W_k$. We say that $W$ satisfies the resonance condition if 
 \[
 \int_{\bbS^{n-1}} Y_{k_{\min}, 0}^3d\sigma \neq 0,
 \]
 where $Y_{k_{\min}, 0}$ is the corresponding spherical harmonic.
\end{assumption}
The Assumption~\ref{assum:pt} can only be satisfied for even $l$. Indeed from the recurrence relation we obtain that for odd $l$ the harmonic $Y_{l,0}$ is an odd function, implying that $Y_{l,0}^3$ is also odd and thus integrates to zero. Calculation for $l=2$ gives
\begin{align*}
Y_{2, 0}(\theta) &= A_0^2 C_{2}^{\frac{n-2}{2}}\left(\cos \theta_{n-1}\right), \\
\int Y_{2, 0}^3d\sigma &= (A_0^2)^3\int_{-1}^1\left( C_{2}^{\frac{n-2}{2}}(t)\right)^3(1-t^2)^{\frac{n-3}{2}}dt = \frac{4(A_0^2)^3(n-2)^3\sqrt{\pi}\Gamma(\frac{n+1}{2})}{(n+2)(n+4)\Gamma(\frac{n}{2}-1)} \neq 0
\end{align*}
and similar for $l=4$
\begin{align*}
\int Y_{4, 0}^3d\sigma &= \frac{(A_0^4)^3(n-2)^3n^4(n^2-4)\sqrt{\pi}\Gamma(\frac{n+5}{2})}{64\Gamma(\frac{n}{2}+6)},
\end{align*}
where $\Gamma$ is the Gamma function and $A_K^l$ are the normalizing coefficients as in \eqref{eq:harmonics}. So in both cases Assumption \ref{assum:pt} is satisfied for a sphere of arbitrary dimension $n \geq 3$. We conjecture that it is satisfied for all even values of $l$. 

\begin{theorem}
\label{th:pt}
Let the interaction potential $W \in C_b$ satisfy Assumption \ref{assum:pt-almost}, then there exists a discontinuous transition point of \eqref{eq:mckean-vlasov} $\gamma_c \in (0, \gamma_\#)$.    
\end{theorem}
\begin{proof}
    By Proposition \ref{prop:existence-pt} the system exhibits a phase transition. We aim to show that the phase transition point does not happen at the point of linear stability $\gamma_c \neq \gamma_\#$, then by Lemma \ref{lemma:tp_lsp} it must be a discontinuous transition point. By Lemma \ref{lemma:competitor} below $\bar\rho$ is not a minimizer of $\calF_{\gamma_\#}$, so the phase transition occurs at some $\gamma_c < \gamma_\#$ and is discontinuous.

\end{proof}

\begin{lemma}
    \label{lemma:competitor}
    Let the interaction potential $W \in C_b$ satisfy Assumption \ref{assum:pt-almost} for some~$u$ and $\delta\in[0,1)$, then there exists a state $\rho_\# \in \calP_{ac}(\bbS^{n-1})$ with lower free energy compared to the uniform state $\bar\rho$ at the point of linear stability $\calF_{\gamma_\#}(\rho_\#) < \calF_{\gamma_\#}(\bar\rho)$.
\end{lemma}
\begin{proof}
    Take $u$ as in the Assumption \ref{assum:pt-general} and denote 
    \[
    \xi = \sign\bra[\bigg]{\int  u^3d\sigma},
    \] we show that there exist $\varepsilon >0$ such that the statement of the lemma holds for the competitor $ \rho_\#$ of form:
    \[
    \rho_\# = \bar\rho(1 + \varepsilon \xi u). 
    \]
    Consider the Taylor expansion of $\calE$ and $S$ around the uniform state:
    \begin{align*}
    \calE(\rho_\#) &= \calE(\bar\rho) +\frac12  \varepsilon^2\|u\|^2  - \frac{1}{6}\xi\bar\rho\varepsilon^3\int  u^3d\sigma + \frac{1}{12}\varepsilon^4\bar\rho \int\frac{u^4}{(1+\varepsilon\xi \tilde u)^3}d\sigma,
    \end{align*}
    for some $\tilde u$ satisfying $ |\tilde u(x)| \in  [0, |u(x)|]$ for all $x\in \bbS^{n-1}$, and
    \begin{align*}
        \calI(\rho_\#) &= \calI(\bar\rho) + \frac12\varepsilon^2\bar\rho^2\int_{\bbS^{n-1}\times \bbS^{n-1}} \sum_{l\in K_{\sharp,\delta}} c_l^2 W(x,y)Y_{l, 0}(x)Y_{l, 0}(y)d\sigma(x)d\sigma(y) \\
        &= \calI(\bar\rho) + \frac12\varepsilon^2 \bar\rho \sum_{l\in K_{\sharp,\delta}}\hat W_l c_l^2 \|Y_{l,0}\|^2 \leq \calI(\bar\rho) - \frac{1}{2\gamma_\sharp}\varepsilon^2 \bar\rho \|u\|^2 %
         +  \frac{1}{2\gamma_\sharp}\varepsilon^2 \bar\rho \delta\|u\|^2.
    \end{align*}
    Using the above and recalling that thanks to the definition of $\gamma_\sharp$ in~\eqref{eq:def:gamma_sharp}, we obtain a cancellation of the second order terms and arrive at the following expression for the free energy functional:
    \begin{align*}
        \calF_{\gamma_\#}(\rho_\#) - \calF_{\gamma_\#}(\bar\rho) &= - \frac{1}{6\gamma_{\#}}\xi\bar\rho\varepsilon^3\int  u^3d\sigma + \frac{1}{12\gamma_{\#}}\varepsilon^4\bar\rho \int\frac{u^4}{(1+\varepsilon\xi \tilde u)^3}d\sigma + \frac{1}{2\gamma_{\#}}\varepsilon^2 \bar\rho \delta\|u\|^2.
    \end{align*}
If $\delta = 0$, we obtain
\[
\calF_{\gamma_\#}(\rho_\#) - \calF_{\gamma_\#}(\bar\rho) = - \frac{1}{6\gamma_{\#}}\bar\rho\varepsilon^3\left(|U_3| - \frac{1}{2}\varepsilon \int\frac{u^4}{(1+\varepsilon\xi \tilde u)^3}d\sigma\right). 
\]
Recall that $|u|_\infty \leq 1$, then taking $\varepsilon < 1/2$ allows us to bound the integral above as
\[
\frac{1}{2} \int\frac{u^4}{(1+\varepsilon\xi \tilde u)^3}d\sigma \leq 4\sigma(\bbS^{n-1}).
\]
Choosing $\varepsilon < \min\left(\frac{1}{2}, \frac{|U_3|}{4\sigma(\bbS^{n-1})}\right)$ we obtain $ \calF_{\gamma_\#}(\rho_\#) - \calF_{\gamma_\#}(\bar\rho) < 0$, and thus the statement of the lemma is satisfied.

In case $\delta >0$, by choosing $\varepsilon = \sqrt{\delta}$ and using $|u|_\infty \leq 1$, we similarly conclude that the energy difference is negative
\[
\calF_{\gamma_\#}(\rho_\#) - \calF_{\gamma_\#}(\bar\rho) \leq - \frac{1}{6\gamma_{\sharp}}\bar\rho\delta^{3/2}\left(|U_3| - 4\delta^{1/2} \sigma(\bbS^{n-1}) - 3\delta^{1/2}\sigma(\bbS^{n-1})\right) < 0,
\]
for $\delta$ chosen sufficiently small. Hence, the statement of the lemma holds for the corresponding $u$.
\end{proof}
\begin{remark}
\label{remark:near-resonance}
    As studied in \cite{carrillo2020long}, existence of a discontinuous phase transition of McKean-Vlasov equation on a torus can be proven under a resonance condition between at least three different components of the Fourier decomposition of $W$. In contrast, on a sphere, due to the structure of the basis functions of $L_2^s$, a single negative mode $\hat W_k$ may be enough for a discontinuous phase transition. %
\end{remark}
\begin{remark}
    We also remark that in the spherical setting we restrict the set of possible competitors to finite linear combinations because the basis functions are not guaranteed to be uniformly bounded. As a result, the residual term in the Taylor expansion can only be made arbitrary small for a finite linear combinations. 
\end{remark}
\begin{remark}[Continuous phase transition]
   As follows from the Lemma~\ref{lemma:tp_lsp}, to show continuity of the phase transition point it is enough to prove that the phase transition happens at the point of linear stability $\gamma_c = \gamma_\#$. In \cite{carrillo2020long}, the authors use a criteria ensuring the uniform growth of the entropy under all perturbations corresponding to the negative components of the kernel decomposition. The difficulty in generalizing this approach to the spherical case lies in the fact that spherical harmonics are not uniformly bounded. 
\end{remark}
\section{Examples}
\label{sec:examples}
We apply our results to several interaction kernels including exponential kernel (transformer model), the Onsager model of liquid crystals and an adaptation of the Hegselmann-Krause model to a high-dimensional sphere. This section heavily relies on the properties of the Gegenbauer polynomials, for which we refer the reader to \cite{abramowitz1968handbook} and their integrals, which can be found in~\cite{gradshteyn2014table}.
\subsection{Noisy Transformers}
Transformer is an architecture of neural networks, which revolutionized the field of natural language processing. It was first introduced in the paper by Vaswani et. al \cite{vaswani2017attention}. In \cite{geshkovski2024mathematical} the authors propose to use an interacting particles perspective to study the behavior of transformers. In particular it is shown that the underlying inference dynamics of transformer model is related to the gradient flow in the space of probability measures on a high dimensional sphere $\calP(\bbS^{n-1})$ of the interaction energy with exponential
kernel $W_\beta$:
\begin{equation}
\label{eq:kernel-transformer}
W_\beta(x, y) := -\frac{1}{\beta}e^{\beta\left<x, y\right>}.
\end{equation}

In this example we consider the system of identical interacting particles on a high-dimensional sphere in the presence of noise
\[
d X_i =  -\nabla_{X_i}\frac{1}{N}\sum_{j=1}^N W_\beta(X_i, X_j) + \sqrt{2\gamma^{-1}}dB_i,
\]
where $\{B_i\}_{i\leq N}$ are independent Brownian motions on a sphere. This system can be seen as a proxy of the transformer model with perturbation of order $\gamma^{-1}$. We call the corresponding McKean-Vlasov equation the \emph{noisy transformer model} and study the structure of the set of invariant measures of this system for various set of parameters $\beta, \gamma \in \bbR_+$.

Clearly $W_\beta$ is bounded and rotationally symmetric, so the results of Section \ref{sec:general} apply. In particular note that $|\partial W_\beta|_\infty = e^\beta$ and $|\partial^2 W_\beta|_\infty = \beta e^\beta$, so using Corollary \ref{cor:convexity-sph} we conclude that for the inverse noise ratio $\gamma < \gamma_o = \frac{(n-2)}{4 \max(\beta e^{\beta}, e^\beta)}$ the uniform state is the unique invariant measure of the noisy transformer model. Since it is unique, the uniform states is the long-term limit of a solution with arbitrary admissible initial conditions. For the transformer model it implies that for the noise with the amplitude larger than $\gamma_o^{-1}$ the model does not preserve any information about the initial conditions, which is strongly undesirable from the application perspective.

We are also able to recover the bifurcation branches on every level of the spherical harmonics, namely:
\begin{proposition}[Noisy transformer on $\bbS^{n-1}$]
    Consider McKean-Vlasov equation \eqref{eq:mckean-vlasov} with the exponential interaction kernel defined in \eqref{eq:kernel-transformer}. Then for every $\beta\in \bbR_+$, for all $k \in \bbN_+$ there exist a bifurcation branch around $(\bar\rho, \gamma_{k})$, where 
    \[
    \gamma_{k} = \frac{\sqrt{\beta^n}}{2^{\frac{n-2}{2}}\Gamma(\frac{n}{2})I_{k+\frac{n-2}{2}}(\beta)}, 
    \]
    and $I_{\alpha}$ is the modified Bessel function. %
\end{proposition}
\begin{proof}
The Rodrigues formula for Gegenbauer polynomials (see~\cite[p. 303]{Andrews_Askey_Roy_1999} or~\cite[p.~143f]{chihara2011introduction}) provides the representation
\begin{equation}\label{eq:Rodrigues}
C_k^\alpha(t) = \frac{(-1)^k\Gamma(\alpha+\frac12)\Gamma(k+2\alpha)}{2^k k!\Gamma(2\alpha)\Gamma(\alpha+k+\frac12)}(1-t^2)^{-\alpha+\frac12}\frac{d^k}{dt^k}\left[(1-t^2)^{k+\alpha-\frac12}\right].
\end{equation}
Using the equality $C_k^\alpha(1) = \frac{\Gamma(k+2\alpha)}{\Gamma(2\alpha)\Gamma(k+1)}$ and the Rodrigues formula~\eqref{eq:Rodrigues}, after integrating by parts $k$ times we obtain the following expression for the spherical harmonics decomposition of $W_\beta$:
    \begin{align*}
\hat W_{\beta,k} &= -\frac{c_{\frac{n-2}{2}}}{\beta C_k^{\frac{n-2}{2}}(1)}\int_{-1}^1e^{\beta t}C^{\frac{n-2}{2}}_k(t)(1-t^2)^{\frac{n-3}{2}}dt \\
&= \frac{(-1)^{k+1}\Gamma(\frac{n}{2})}{2^k\beta\sqrt{\pi}\Gamma(k+\frac{n-1}{2})}\int_{-1}^1 e^{\beta t}\frac{d^k}{dt^k}\left[(1-t^2)^{k+\frac{n-3}{2}}\right] dt \\
&= \frac{(-1)^{2k+1}\Gamma(\frac{n}{2})\beta^{k-1}}{2^k\sqrt{\pi}\Gamma(k+\frac{n-1}{2})}\int_{-1}^1 e^{\beta t}(1-t^2)^{k+\frac{n-3}{2}} dt = -2^{\frac{n-2}{2}}\beta^{-\frac{n}{2}}\Gamma\left(\frac{n}{2}\right)I_{k+\frac{n-2}{2}}(\beta).
\end{align*}
Modified Bessel functions $I_\alpha(x)$ are decreasing in $\alpha$ for any fixed $x \in \bbR_+$, so all the components of the spherical harmonics decomposition of the exponential kernel are distinct. Applying the Theorem \ref{th:bifurcations} we get the result.
\end{proof}
\begin{remark}[Phase transition]
    Since modified Bessel functions $I_\alpha(x)$ are decreasing in $\alpha$ for any fixed $x \in \bbR_+$, so for all admissible $\beta$ the smallest component of the spherical harmonics  decomposition of the interaction kernel $W_\beta$ corresponds to $k_{\min} = k_1$. At the same time, corresponding harmonic $Y_{1, 0}$ is odd and thus it's cube integrates to zero, so the resonance Assumption \ref{assum:pt} is not satisfied and a finer analysis is required to establish the type of the phase transition of the noisy transformer model.
\end{remark}
\begin{remark}
    As studied in \cite{bruno2024emergence, geshkovski2024dynamic}, in the noiseless case $\gamma = \infty$, the transformer model exhibits metastable behavior. We expect the noisy model to have similar dynamics for large $\gamma$. Since the number of bifurcations points growths with the inverse temperature $\gamma$, we expect the model to have more saddles and more evident metastable behavior as $\gamma$ increases. 
\end{remark}
\subsection{Onsager model}\label{ssec:Onsager}
Onsager model of liquid crystals describes the behaviour of the liquid crystals materials in so-called nematic phase. Liquid crystals are assumed to consist of oriented interacting particles, and during nematic phase the main contribution of the interaction energy is coming from the orientation of the molecules (not their spatial coordinates), and thus the model is naturally posed on $\bbS^{n-1}$. Onsager interaction kernel has the form
\begin{equation}
 \label{eq:kernel-onsager}   
W(x, y) := \sqrt{1 - \left<x, y\right>^2}.
\end{equation}
Note that in this case correlated molecules $\left<x, y\right> = 1$ have the same contribution as anticorelated ones $\left<x, y\right> = -1$. The noise parameter $\gamma$ in this case plays the role of inverse temperature or particles 'mobility'. Existence of solutions and bifurcation branches of the Onsager model on~$\bbS^2$ have been established in \cite{WachsmuthThesis06,vollmer2018bifurcation}. In this Section we generalize the result to the sphere of arbitrary dimension $\bbS^{n-1}$ and characterize the phase transition of the corresponding model.

   \begin{proposition}[Onsager model on $\bbS^{n-1}$]
   Consider the McKean-Vlasov equation with the interaction kernel \eqref{eq:kernel-onsager}, then 
    \begin{itemize}
        \item the model undergoes a discontinuous phase transition,
        \item for every $k = 2l, \ l\in \bbN$ there exists a bifurcation branch around $(\bar\rho, \gamma_{2l})$,
        where
        \[
        \gamma_{2l} = s_n^{-1}\frac{\Gamma(l+\frac{n+1}{2})\Gamma(l+n-2)}{ \Gamma(l- \frac12 )\Gamma(l+ \frac{n-2}{2})},
        \]
        where $s_n = \frac{(n-2)\Gamma(\frac{n}{2})\Gamma(n-2)}{4\sqrt{\pi}\Gamma(\frac{n-1}{2})}$
    \end{itemize}
   \end{proposition}
   \begin{proof}
       First of all note that both the kernel and the marginal of the spherical measure $\sigma$ are even functions of $\theta_{n-1}$. At the same time, since odd Gegenbauer polynomials are odd functions, corresponding coefficients of the spherical harmonics decomposition of the Onsager kernel are zero, namely $\hat W_{2l+1} = 0$ for all $l\in\bbN$. To calculate the even components of the decomposition we use the representation of the Gegenbauer polynomials in the monomial basis:
       \begin{align*}
       C_{k}^{\frac{n-2}{2}}(x)
    = \sum_{j=0}^{\lfloor k/2 \rfloor}(-1)^j \frac{\Gamma(k - j + \frac{n}{2} - 1)}{\Gamma(\frac{n}{2} - 1)j!(k - 2j)!}(2t)^{k - 2j}.
       \end{align*}
       Integrating the monomials we obtain
       \begin{align*}
            \int_{-1}^{1} t^{2m}\left( 1 - t^2 \right)^{(n - 2) / 2}~dt= \int_{0}^{\pi} \cos^{2m}\theta \sin^{n-1}\theta ~d\theta = \frac{\Gamma(m+ \frac{1}{ 2})\Gamma(\frac{n}{2})}{\Gamma(m +\frac{n+1}{2})},
       \end{align*}
       and combining the above formulas we get the following expression for the even coefficients $\hat W_{2k}:$
       \begin{align*}
           \hat W_{2k} &= \frac{c_{\frac{n-2}{2}}}{C_{2k}^{\frac{n-2}{2}}(1)}\int_{-1}^1\sum_{j=0}^{k}(-1)^j \frac{\Gamma(2k - j + \frac{n}{2} - 1)}{\Gamma(\frac{n}{2} - 1)j!(2k - 2j)!}4^{k-j}t^{2(k - j)}(1-t^2)^\frac{n-2}{2}dt \\
           &= \frac{(\frac{n}{2} - 1)\Gamma(\frac{n}{2})\Gamma(n-2)\Gamma(k+1)}{\Gamma(\frac{n-1}{2})\Gamma(k+n-2)} \sum_{j=0}^{k}(-1)^j \frac{\Gamma(2k - j + \frac{n}{2} - 1)}{j!(k - j)!\Gamma(k - j + \frac{n+ 1}{2})} \\
           &=-\frac{\sqrt{\pi}(\frac{n}{2}-1)\Gamma(\frac{n}{2})\Gamma(n-2)}{2\Gamma(\frac{n-1}{2})\Gamma(-\frac{3}{2})\Gamma(\frac{5}{2})} \times \frac{\Gamma(k-\frac{1}{2})\Gamma(k+\frac{n}{2}-1)}{\Gamma(k+\frac{n+1}{2})\Gamma(k+n-2)} \\
           &= -s_n \cdot \frac{\Gamma(k-\frac{1}{2})\Gamma(k+\frac{n}{2}-1)}{\Gamma(k+\frac{n+1}{2})\Gamma(k+n-2)},
       \end{align*}
       where $s_n$ is a dimension-dependent constant. Note that $s_n >0$ for $n\geq 3$ since all the components are positive in this case. Also note that the coefficients are decreasing in absolute value in $k$, for any $n\geq 3$. Namely we have
       \begin{align*}
       \frac{\hat W_{2k+2}}{\hat W_{2k}} &= \frac{\Gamma(k+\frac12 )\Gamma(k+\frac{n}{2})\Gamma(k+n-2)\Gamma(k+\frac{n+1}{2})}{\Gamma(k+1+\frac{n+1}{2})\Gamma(k+n-1)\Gamma(k- \frac12 )\Gamma(k - 1+\frac{n}{2})}\\
       &= \frac{(k-\frac12)(k+\frac{n}{2} - 1)}{(k +\frac{n+1}{2})(k+n-2)} = \frac{k-\frac12}{k+n-2} \times  \frac{k+\frac{n-2}{2}}{k+\frac{n+1}{2}} <1,
       \end{align*}
       implying that for any dimension $n\geq 3$ the minimal component corresponds to $k =2$. As shown in Section \ref{ssec:transition}, the second spherical harmonic $Y_{2,0}$ satisfies the self-resonance Condition \ref{assum:pt}, and thus the system has a discontinuous transition point. Moreover, since all the components are distinct, applying Theorem \ref{th:bifurcations} we recover all the bifurcation branches of the corresponding free energy functional.
   \end{proof}
   \begin{remark}
       For $n = 3$ we get $s_n = 1/8$ so the bifurcation points correspond to the values 
       \[
       \gamma_{2l} = \frac{8\Gamma(l+2)\Gamma(l+1)}{ \Gamma(l- \frac12 )\Gamma(l+ \frac{1}{2})},
       \]
       and we recover the result of \cite{vollmer2018bifurcation}. 
   \end{remark}
\subsection{Opinion dynamics}\label{ssec:opinion}
We introduce a spherical analogue of the Hegselmann-Krause model of opinion dynamics proposed in \cite{HegselmannKrause2002}. In this model, posed in $\bbR$, the particles (or 'agents') attract if they are in a close proximity of each other and have no effect on each other otherwise. We propose using the following interaction kernel as a spherical alternative of the Hegselmann-Krause model:
\begin{equation}
\label{eq:kernel-opinion}   
W_p(x, y) = -(1+\left<x, y\right>)^p = -(2 - (1-\left<x, y\right>))^p,
\end{equation}
where $p \in \bbR_+$. Note that the expression $(1-\left<x, y\right>)$ plays the role of the 'distance' between two vectors, and thus the kernel is an alternative to the rational kernel of the following form $W(x,y) = -(2 - \|x- y\|)_+^p$ on $\bbR$. Note, that localization in this model corresponds to large values of $p$.

The kernel is an element of $C_b$ for any $p\geq  0$ and an element of $H^1$ for $p> \frac12 -\frac{n-3}{2}$, so the corresponding existence (\ref{th:minimizers}) and regularity (\ref{prop:equivalence}) results apply. For $p\geq 2$, the second derivative is also bounded and we conclude that similar to the noisy transformer model, for the inverse noise ratio $\gamma <\gamma_o = \frac{(n-2)}{4\max(2^{p-1}p, 2^{p-2}p(p-1))}$ the free energy is geodesically convex and thus the model obtains a unique invariant measure, the uniform state $\bar\rho$. We can also characterize the set of bifurcation branches of the model on an arbitrary-dimensional sphere.
\begin{proposition}[Bifurcations for the opinion dynamics model on $\bbS^{n-1}$]
Consider the McKean-Vlasov equation with rational interaction kernel as introduced in \eqref{eq:kernel-opinion}, then 
\begin{itemize}
    \item for $p\in \bbN$ model has at most $p$ bifurcation branches,
    \item if $p\in \bbN$ and $p\leq n-1$, then for every $k \in \{1...p\}$ there exist a bifurcation branch around the uniform state
    \item for $p\in \bbR_+\backslash \bbN$ the model has a countable number of bifurcation branches, namely for every $k \in \bbN$ such that $\frac{(-1)^k\Gamma(k-p)}{\Gamma(-1-p)} < 0$ there exists a bifurcation branch at the $k$-th level of spherical harmonics basis
    \item if there exist a bifurcation branch at the $k$-th level, the corresponding $\gamma_k$ takes the form
    \[
    \gamma_k = \frac{\sqrt{\pi}\Gamma(p+1-k)\Gamma(n+k-1)}{2^{n-2+p}\Gamma(\frac{n}{2})\Gamma(\frac{n-1}{2}+p)\Gamma(p+1)}
    \]
   \end{itemize} 
\end{proposition}
\begin{proof}
First note that for $p\in \bbN$ the kernel is of a polynomial type and thus can be exactly represented by a linear combination of the first $p$ Gegenbauer polynomials, implying that $\hat W_{p, k} = 0$ for $k > p$. In general case, the spherical harmonics decomposition of $W_p$ takes the form
\begin{align*}
    \hat W_{p, k} &= -\frac{c_{\frac{n-2}{2}}}{C_k^{\frac{n-2}{2}}(1)}\int_{-1}^1(1+t)^pC^{\frac{n-2}{2}}_k(t)(1-t^2)^{\frac{n-3}{2}}dt \\
    &=-\frac{\Gamma(\frac{n}{2})\Gamma(n-2)\Gamma(k+1)}{\sqrt{\pi}\Gamma(\frac{n-1}{2})\Gamma(k+n-2)} \int_{-1}^1C^{\frac{n-2}{2}}_k(t)(1-t)^{\frac{n-3}{2}}(1+t)^{\frac{n-3}{2}+p}dt \\
    &= -\frac{2^{n-2+p}\Gamma(\frac{n}{2})\Gamma(\frac{n-1}{2}+p)\Gamma(p+1)}{\sqrt{\pi}\Gamma(p+1-k)\Gamma(n+k-1)} \sim -\frac{1}{\Gamma(p+1-k)\Gamma(n+k-1)},
\end{align*}
where we used the value of the definite integral from \cite[p. 795]{gradshteyn2014table}. Using the properties of Gamma function we then obtain
\[
-\frac{1}{\Gamma(p+1-k)\Gamma(n+k-1)}= \frac{(-1)^k\Gamma(k-p)}{\Gamma(-1-p)\Gamma(p+2)\Gamma(n+k-1)},
\]
and thus for $k> p$ the sign of $\hat W_{p, k}$ is $\frac{(-1)^k}{\Gamma(-1-p)}$, implying that either every even or every odd coefficient for $k>p$ is negative and all the components are distinct. Applying the Theorem \ref{th:bifurcations} we get the result. Finally note that for $p\in \bbN$ all nonzero components of the spherical harmonics decomposition are negative. Calculating the ratio we obtain
\begin{equation}\label{eq:HK:W:ratio}
\frac{\hat W_{p, k+1}}{\hat W_{p, k}} = \frac{\Gamma(p+1 - k)\Gamma(n+k-1)}{\Gamma(p-k)\Gamma(n+k)} = \frac{p-k}{n-1 +k },
\end{equation}
so for $p \leq n-1$, the ratio satisfies $\frac{\hat W_{p, k+1}}{\hat W_{p, k}} < 1$ for $k\geq 1$, implying that all the coefficients are distinct and thus, by theorem \ref{th:bifurcations} there exists a bifurcation branch for every $k \in \{1...p\}$.
\end{proof}
First of all, the above result shows that the sign of the coefficient of the spherical harmonics decomposition of $W_p$ does not depend on the dimensionality of the sphere but only on the parameter $p$. In particular, for $k=1$ the above criterion gives
\[
-\frac{\Gamma(1-p)}{\Gamma(-1-p)} = -(-p)(-1-p) = -p(1+p) < 0,
\]
so the bifurcation branch might exist for any positive $p$. Similar calculation for $k=2$ gives
\[
\frac{\Gamma(2-p)}{\Gamma(-1-p)} = (1-p)p(1+p) <0,
\]
implying that the branch can exist only if $p>1$. Analogous analysis can be conducted for any $k\in \bbN_+$. 

\begin{proposition}[Discontinuous phase transition for opinion dynamics]\label{prop:HK:discont}
    The McKean-Vlasov equation with rational interaction kernel as introduced in \eqref{eq:kernel-opinion} has a discontinuous phase transition on $\bbS^{n-1}$ for some $\gamma_c<\gamma_{\#}$ provided that $p=n+2$.
\end{proposition}
\begin{proof}
    We verify Assumption~\ref{assum:pt} for $k_{\min}=2$. So, we have to ensure that indeed $k=2$ is the most negative mode. This can be observed immediately from~\eqref{eq:HK:W:ratio} with the choice $p=n+2$, which then becomes $\hat W_{n+2,k+1}/\hat W_{n+2,k}=(n+2-k)/(n-1+k)$. Indeed, for $k=1$, we have $\hat W_{n+2,k+1}/\hat W_{n+2,k}=n+1/n > 1$ and for $k\geq 2$, we get 
    \[
    \frac{\hat W_{n+2,k+1}}{\hat W_{n+2,k}}=\frac{n-1+k+3-2k}{n-1+k}=1-
\frac{2k-3}{n-1+k}<1.
    \]
    Hence, $k_{\min}=2$ and we obtain a discontinuous phase transition from Theorem~\ref{th:pt}.
\end{proof}

\subsection{Localized spherical Gaussian kernel}\label{ssec:localized}
The Hegselmann-Krause model of opinion dynamics on $\bbR^n$ has a localized nature, while in our spherical analogue introduced in the previous section we only found a very special range of kernels for which we can prove a discontinuous phase transition (see Proposition~\ref{prop:HK:discont})
For this reason, we study a system with a localized kernel, where Assumption~\ref{assum:pt} is verifiable.
A canonical choice is the McKean-Vlasov equation with the exact hyperspherical heat kernel $W_\varepsilon(x, y) = - u(\varepsilon, x, y)$, where $u(\varepsilon, x, y)$ is the fundamental solution on $\bbS^{n-1}$. As follows from \cite{zhao2018exact}, $W_\varepsilon$ has an explicit representation in terms of the Gegenbauer polynomials:
\begin{equation}
    \label{eq:heat-kernel}
    W_\varepsilon(x,y) = -\sum_k e^{-k(k+n-2)\varepsilon}\frac{2k+n-2}{n-2}\frac{\Gamma(\frac{n}{2})}{2\sqrt{\pi^n}}C_k^{\frac{n-2}{2}}(\skp{x,y}),
\end{equation}
using Lemma \ref{lemma:gegegnbauer-decomposition} we can recover the bifurcation branches of the corresponding McKean-Vlasov equation.
\begin{proposition}[Localized interaction model on $\bbS^{n-1}$]
\label{prop:localized}
    Consider the McKean-Vlasov equation with the localized interaction kernel as in \eqref{eq:heat-kernel}, then for every $k\in \bbN$ there exists a bifurcation branch around $(\bar \rho, \gamma_k)$, where $\gamma_k$ has the form
    \[
    \gamma_k = \frac{2\sqrt{\pi^n}}{\Gamma(\frac{n}{2})}e^{k(k+n-2)\varepsilon}.
    \]

\end{proposition}
\begin{proof}
    Using Lemma \ref{lemma:gegegnbauer-decomposition} we conclude that the spherical harmonics decomposition of $W_\varepsilon$ takes the form
    \[
    \hat W_{\varepsilon, k} = -\frac{\Gamma(\frac{n}{2})}{2\sqrt{\pi^n}}e^{-k(k+n-2)\varepsilon}.
    \]
    Note that all the coefficients $\hat W_{\varepsilon, k}$ are distinct for arbitrary $\varepsilon$ and arbitrary $n$. Applying Theorem~\ref{th:bifurcations} we recover all the bifurcation branches of the corresponding system.
\end{proof}
\begin{proposition}[Discontinuous phase transition of the localized interaction model]
Consider the McKean-Vlasov equation with the localized interaction kernel \eqref{eq:heat-kernel}, then there exists $\varepsilon_0 > 0$ such that for all $\varepsilon \in (0, \varepsilon_0)$ the model undergoes a discontinuous phase transition.
\end{proposition}
\begin{proof}
 It follows from the proposition above that all the components of the spherical harmonics decomposition are different and $k_{\min} = 1$. However, for small $\varepsilon$ the difference between the first and the second components decays with $\varepsilon$:
 \[
 \hat W_{\varepsilon, 2} - \hat W_{\varepsilon, 1} \sim (n+1)\varepsilon
 \]
 and thus we can show that for small enough $\varepsilon$ the model satisfies the relaxed resonance condition \ref{assum:pt-almost}. Making this observation rigorous we obtain
 \[
 \hat W_{\varepsilon, 2} - \hat W_{\varepsilon, 1} = \frac{\Gamma(\frac{n}{2})}{2\sqrt{\pi^n}}\left(e^{-(n-1)\varepsilon}  - e^{-2n\varepsilon}\right) = \frac{\Gamma(\frac{n}{2})e^{-(n-1)\varepsilon}}{2\sqrt{\pi^n}} \left(1  - e^{-(n+1)\varepsilon}\right) \leq \frac{3\varepsilon(n+1)\Gamma(\frac{n}{2})}{4\sqrt{\pi^n}},
 \]
 provided that $\varepsilon(n+1) < 1$. Taking $u = \frac{Y_{2, 0}}{|Y_{2, 0}|_\infty}$, we then can check that for $\varepsilon < \frac{C_0\Gamma\left(\frac{n}{2}\right)U_3^2}{(n-1)\sqrt{\pi^n}} =\varepsilon_0$, where $C_0$ is a constant independent of $n$, the bandwidth $\delta_\varepsilon = \hat W_{\varepsilon, 2} - \hat W_{\varepsilon, 1}$ is sufficiently small for $W_\varepsilon$ to satisfy Condition \ref{assum:pt-almost}. Thus, according to Theorem~\ref{th:pt}, the model undergoes a discontinuous phase transition for any $\varepsilon < \varepsilon_0$.
\end{proof}

\appendix

\section{Differential forms}
\label{sec:geometry}
In this section we give some background information on Riemannian geometry and define the differential operators used in this paper. We consider a smooth compact Riemannian manifold $(\calM, g)$ without boundary with metric $g$ which assigns a positive-definite quadratic form on the tangent bundle $g_x: T_x\calM \times T_x\calM \to \bbR_+$ to any point $p\in \calM$. For a smooth function $\gamma: [0,1] \to \calM$ we define the curve-length as
\[
L(\gamma) := \int_0^1 \sqrt{p_{\gamma(s)}[\gamma'(s),\gamma'(s)]}ds,
\]
where 
\[
\gamma'(s) = \lim_{\delta \to 0}\frac{1}{2\delta} [\gamma(s+\delta) - \gamma(s-\delta)] \in T_{\gamma(s)}\calM,
\]
for $s \in (0,1)$, and the right or left limit for the end points. Then the distance between any two points $x, y \in \calM$ is defined as
\[
\dist(x,y) = \inf_{\gamma \in \Gamma_{x,y}}L(\gamma),
\]
where $\Gamma_{x,y}$ is a set of all smooth curves $\gamma$ satisfying $\gamma(0)= x, \ \gamma(1)= y$. Given a connection $\nabla$ on $\calM$, a constant-speed geodesic $\gamma$ is a curve satisfying the zero-acceleration condition
\begin{equation}
    \label{eq:geodesic}
\nabla_{\gamma'(s)}\gamma'(s) = 0.
\end{equation}
In this work we will always consider Levi-Civita connection, the unique for every manifold torsion-free and metric-preserving connection. For two smooth vector fields $X, Y$ the map $\nabla_XY: \calM \to  T\calM$ is called a covariant derivative of $Y$ in the direction of $X$ and it's evaluation at $\xi\in\calM$ shows the change of vector $Y_\xi = Y(\xi)$ in the direction of $X_\xi= X(\xi)$. %

 For any fixed connection on $\calM$, for every point $x\in \calM$ and a tangent vector $v\in T_x\calM$ there exists a unique
geodesic $\gamma_{x,v}: [0,1] \to \calM$ with initial condition $\gamma(0) = x, \ \gamma'(0) = v$. Then the exponential map is the end point of such a curve:
\[
\exp_x(v) = \gamma_{x,v}(1).
\]
For a smooth function $f: \calM \to \bbR$ it's differential at a point $x\in\calM$ is a linear map $df_x: T_x\calM \to \bbR$ such that for any smooth curve satisfying $\gamma(0) = x, \ \gamma'(0) = v$ it holds that
\[
df_x(\gamma'(0)) = (f\circ \gamma)'(0),
\]
where the expression on the right hand side $f\circ \gamma$ is a curve in $\bbR$. The gradient of a smooth function $f: \calM \to \bbR$ is a vector field $\grad f$ which for any vector field $Z$ on $\calM$ and any point $x\in \calM$ satisfies
\[
g_x((\grad f)_x, Z_x) = df_x(Z_x).
\]
\begin{example}
    On a unit sphere $\calM = \bbS^{n-1}$ equipped with distance $\dist(x, y) = \arccos(\left<x, y\right>)$ the manifold gradient $\grad_{\bbS^{n-1}} f$ in Euclidean coordinates is equal to the projection of the Euclidean gradient onto the tangent space at $x$:
    \[
    \grad_{\bbS^{n-1}} f_x = \nabla_{\bbR^n} f_x -  \left<\nabla_{\bbR^n} f_x, x\right> x,
    \]
    where $\left<\cdot, \cdot \right>$ is a Euclidean scalar product and $\nabla_{\bbR^n} f_x = \left(\frac{\partial f(x)}{\partial x_1}, ... \frac{\partial f(x)}{\partial x_n} \right)$.
\end{example}
\begin{remark}
    In a special case when the metric $g$ is induced by a scalar product in some Euclidean space $\bbR^{n'}$ for $n'>n$, the scalar product satisfies
    \[
    g(X, Y) = \left<X, P_{T\calM}Y\right>_{\bbR^{n'}},
    \]
    and as a result the manifold gradient is a projection of the Euclidean gradient onto the tangent space
    $\grad_\calM f = P_{T\calM}\grad_{\bbR^{n'}}f$.
\end{remark}
Divergence of a smooth vector-field $X$ on a manifold is a trace of the covariant derivative $\nabla X$ with Levi-Civita connection:
\[
\divr X := \tr (\nabla X),
\]
where $\nabla X$ is an object which for every smooth vector field $Y$ satisfies $\nabla X(Y) = \nabla_Y X$ 
in particular let $\{e_i\}$ be an orthonormal basis of the tangent bundle $T\calM$, then
\[
\divr X = \sum_i \left<\nabla_{e_i} X, e_i\right> = \sum_i g(\nabla_{e_i} X, e_i).
\]
For an $n$-dimensional Riemannian manifold we can define a unique volume measure $m$ which in local coordinates takes the form
\[
dm = \sqrt{\det g_{ij}}dx,
\]
where $g_{ij}$ is the metric tensor in local coordinates and $dx$ is the Lebesgue volume element in $\bbR^n$. As a result for any compact manifold without boundary $(\calM, g)$ we get the following integration by part rule
\[
\int \phi \cdot \divr X dm = -\int g(\grad \phi, X )dm
\]
for any $\phi \in C^\infty(\calM)$.

Laplace-Beltrami operator is a generalization of the Laplace operator to the manifold setting, namely for any smooth function $f: \calM \to \bbR$ such that $\grad f$ is a smooth vector field the action of the Laplace-Beltrami operator is defined as
\[
\Delta f := \divr (\grad f).
\]
\begin{example}[Corollary 1.4.3 {\cite{dai2013approximation}}]
    On a unit sphere $\calM = \bbS^{n-1}$ equipped with distance $\dist(x, y) = \arccos(\left<x, y\right>)$ the Laplace-Beltrami operator $\Delta f$ is equal to the Euclidean Laplacian of the function $\tilde f : \bbR^n \to \bbR$ defined as $\tilde f (x) = f(x/\|x\|)$:
    \[
    \Delta_{\bbS^{n-1}} f = \Delta_{\bbR^n} \tilde f, 
    \]
    where $\Delta_{\bbR^n} = \sum_i \frac{\partial^2}{\partial x_i^2}$.
\end{example}

\begin{remark}(Laplace-Beltrami is a generator of Brownian motion)

Consider a smooth Riemannian manifold $\calM$ and let $p(x, y, t)$ for $x, y \in \calM,  \ t \in [0,T)$ be the (unique) solution of the following initial value problem:
\begin{align*}
    u_t - \Delta u &= 0, \\
    \lim_{t\to +0} u(x, t) &= \delta_y(x), \\ 
    p(y, x, t) &= u(x, t).
\end{align*}
Define the $\calM$-valued stochastic process $B_t$ with the transition probability $\bbP(B_t = x| B_0 = y) = p(y, x, t)$, then one can show that the Laplace-Beltrami is a generator of $B_t$, namely for every admissible $f$ it holds that
\[
\bbE f(B_t) = \bbE f(B_0) + \frac12\int_0^t \bbE\Delta f(B_s)ds.
\]
We will call $B_t$ a Brownian motion on $\calM$. 
\end{remark}

\subsection{Sobolev spaces on manifolds} \label{ssec:SobolevMfds}

We define the $H^1(\calM)$ on a manifold $\calM$ as the space of all functions $f\in L_2(\calM)$ such that the gradient $\nabla f$ exists $m$-almost everywhere and is an element of $L_2(T\calM)$:
\[
H^1(\calM) := \{f: \calM \to \bbR: \int_\calM|f|^2dm < \infty, \ \int_\calM g(\nabla f, \nabla f)dm < \infty\},
\]
equipped with the norm
\[
\|f\|_{H^1(\calM)} = \int_\calM |f|^2dm + \int_\calM g(\grad f , \grad f)dm.
\]
Higher order Sobolev spaces can be defined in terms of the higher order covariant derivatives, we refer the reader to \cite{hebey1996sobolev} for details. Similar to $\bbR^d$, Sobolev spaces on a manifold can also be defined as a closure of the space of smooth functions with respect to the corresponding measure. On the equivalence between two definitions see also \cite{chan2024meyers}.

For any $f\in H^1(\calM)$ we say that $u: \calM \to \bbR$ satisfies $u = \Delta f$ in a weak sense if for any $\phi \in C^\infty(\calM)$ we have
\[
\int \phi \cdot u dm = -\int g(\grad \phi, \grad f )dm.
\]
Finally, note that by definition $H^1(\calM) \subseteq L_2(\calM)$. %

\subsection{Product manifolds}\label{ssec:ProductMfds}

When working with a product space $\calM\times \calM$ we will consider the following product distance $\dist_{\calM \times \calM}(x,y)^2 = \dist_{\calM}(x_1,y_1)^2 + \dist_{\calM}(x_2,y_2)^2$, which implies that at any point there exist local coordinates such that the metric tensor $g_{\calM \times \calM}$ has the block-diagonal form and the scalar product satisfies $g_{\calM\times\calM}(X, Y) = g_{\calM}(X_1, Y_1) + g_{\calM}(X_2, Y_2)$. Then the gradient can also be decomposed into two components $\nabla f(x)= (\nabla_{x_1}f, \nabla_{x_2}f)$ and the volume measure is a product measure $dm_{\calM\times \calM}(x)=dm(x_1)dm(x_2)$.

\section{\texorpdfstring{$\Gamma$}{Gamma}-convergence}\label{appendix:Gamma}
In this section we give the background information on $\Gamma$-convergence, we refer the reader to \cite{braides2006handbook} for the proofs.
\begin{definition}[$\Gamma$-convergence]
    A sequence of functionals $F_n: X \to \bbR$ on some metric space $(X,d)$ is said to converge to a functional $F:X \to \bbR$ in the sense of $\Gamma$-convergence, denoted by $F_n \stackrel{\Gamma}{\to} F$, if
the following two conditions are satisfied:
\begin{itemize}
    \item \emph{(Liminf inequality)} For every convergent sequence $x_n \to \hat x$ it holds
    \[
    F(\hat x) \leq \liminf_{n\to\infty} F_n(x_n),
    \]
    \item \emph{(Limsup inequality)} For every $\hat x \in X$ there exists a convergent sequence $x_n \to x$ such that 
    \[
    F(\hat x) \geq \limsup_{n\to\infty} F_n(x_n).
    \]    
\end{itemize}
\end{definition}
Note that $\Gamma$-convergence depends on the topology of the underlying space $X$. Remarkably, a constant sequence $F_n =F$ converges to $F$ in the sense of  $\Gamma$-convergence only if $F_n$ is a lower-semicontinuous functional. In general case, $\Gamma$-limit of a constant sequence $F_n = F$ is the lower-semicontinuous envelope of $F$.
\begin{definition}[Lower-semicontinuous envelope]
For a functional $F: X\to \bbR$ defined on a metric space $(X, d)$ we define it's \emph{lower-semicontinuous envelope} $\text{lsc}(F): X\to \bbR$ as the largest lower-semicontinuous functional not exceeding $F$:
\[
\text{lsc}(F) := \sup\{ G: \ G \leq F, \ G : X \to \bbR \text{ is $d$-lower semi-continuous} \}.
\]    
\end{definition}
We will use the following proposition.
\begin{proposition}[$\Gamma$-limit of decreasing sequence]
\label{prop:gamma-decreasing}
    Let $(F_n)_{n\in\bbN}$ be a decreasing sequence of functionals and let $F_n \to F$ pointwise, then $F_n \stackrel{\Gamma}{\to} \text{\emph{lsc}}(F)$.
\end{proposition}
$\Gamma$-convergence is an important tool for the characterization of the minimizers of the limiting functionals. The fundamental Theorem of $\Gamma$-convergence establishes necessary conditions for the convergence of minimizers for an equi-coercive sequence of functionals. 
\begin{definition}[Equicoercivity]
    A sequence of functionals $(F_n)_{n\in\bbN}$ is equi-coercive if for all $\xi \in \bbR$ there exist a compact set $K_\xi\subseteq X$ such that $\{x: F_n(x) < \xi\} \subset K_\xi$ for all $n \in \bbN$.
\end{definition}
\begin{theorem}[Fundamental theorem of $\Gamma$-convergence]
    \label{th:gamma-coonvergence}
    Let $(F_n)_{n\in\bbN}$ be an equi-coercive sequence of functionals and let $F_n \stackrel{\Gamma}{\to} F$, such that $\inf F = \alpha_0$, then 
    \begin{enumerate}    
        \item $\lim_{n\to\infty} \inf F_n =\alpha_0$,
        \item if $x_n \to \hat x$ and $F_n(x_n)\to \alpha_0$, then $\hat x$ is a minimizer of $F$,
        \item every sequence $(x_n)_{n\in\bbN}$ satisfying $F_n(x_n)\to \alpha_0$ has a convergent subsequence $x_{n_k} \to \hat x$ and every accumulation point $\hat x$ is a minimizer of $F$.
\end{enumerate}
\end{theorem}

\section{Proofs of auxilary Lemmas for Section \ref{sec:bifurcations}}
\label{appendix:bifurcations}
\begin{proof}[Proof of Lemma \ref{lemma:G}]
\textbf{Uniform convergence of $G(w, \gamma)$.} We begin by showing the uniform in $\gamma$ convergence to zero around zero. Let
\begin{align*}
    A_\gamma(w)&:=\frac{1}{Z(\gamma, w)} e^{-\gamma W* (\bar \rho + w)} = B_\gamma(w)C_\gamma(w), \quad \text{where}\\
    B_\gamma(w) &:= \frac{1}{Z(\gamma, w)}, \\
    C_\gamma(w) &:= e^{-\gamma_0 W* (\bar \rho + w)}.
\end{align*}
Using Taylor expansion we get the estimates for $ B_\gamma(w)$ and $C_\gamma(w)$ for $\|w\|_{L_2} \to 0$:
\begin{align*}
    B_\gamma(w) &= \frac{1}{Z(\gamma, 0) + (Z(\gamma, w) - Z(\gamma, 0))}= \frac{1}{Z(\gamma, 0)}\left(1 + \frac{(Z(\gamma, 0) - Z(\gamma, w))}{Z(\gamma, 0)} + f_1(w)\right) \\
    &= \frac{1}{Z(\gamma, 0)}\left(1 + \bar\rho\int\left(-\gamma W*w + \frac{1}{2}e^{-\gamma\xi_x}(\gamma W*w)^2\right)d\sigma + f_1(w)\right), \\
    C_\gamma(w) &=e^{-\gamma W*\bar\rho}\left(1 -\gamma W* w + \frac{1}{2}e^{-\gamma\xi_x}(\gamma W*w))^2\right),
\end{align*}
where $\xi_x$ is the point of estimation of the second derivative for the Lagrange form  of the remainder, by definition we know that for every $x\in\bbS$ the absolute value of $\xi_x$ is bounded by the value of convolution at $x:$ $|\xi_x| \leq (W*w)(x)$. In Lemma \ref{lemma:gibbs-H1} we have shown the uniform bound on the convolution $\|W*w\|_\infty \leq \|W\|_\infty\|w\|_{L_1}$, and since $\|w\|_{L_2} <\infty$ and the sphere is a compact domain we can bound the $L_1$-norm of $w$ in terms of $L_2$-norm $\|w\|_{L_1} \leq \|w\|_{L_2}^2 + \sigma(\bbS^{n-1})$. As a result, we conclude that $|\xi_x|$ is uniformly bounded and so is the exponential pre-factor $|e^{-\gamma\xi_x}| < C_\xi$. 

We use this argument to estimate the residual term $f_1(w)$, which by Taylor's theorem can be bounded by
\[
|f_1(w)|\leq \beta_f\left(\bar\rho\int\left(-\gamma W*w + \frac{1}{2}e^{-\gamma\xi_x}(\gamma W*w)^2\right)d\sigma\right),
\]
for some $\beta_f(x) =O(x^2)$ as $x \to 0$. At the same time estimating the square of the argument of $\beta_f$, by Cauchy-Schwartz inequality we obtain 
\begin{align}
\label{eq:long-bound}
\MoveEqLeft\left(\int\left(-\gamma W*w + \frac{1}{2}e^{-\gamma\xi_x}(\gamma W*w)^2\right)d\sigma\right)^2 \leq \gamma^2\|W\|^2_\infty\|w\|_{L_1}^2 +\frac14C_\xi^2\gamma^2 \|W\|^4_\infty\|w\|_{L_1}^4 \\
&\qquad \leq \gamma^2(C_1 \|w\|^2_{L_2} + C_2\|w\|^4_{L_2}) = O(\|w\|^2_{L_2}),\notag
\end{align}
uniformly in $\gamma$ on any neighborhood $(\gamma-\delta, \gamma+\delta)\subset \bbR_+$.
As follows from the bound \eqref{eq:long-bound} and the uniform bound obtained in Lemma \ref{lemma:G}, both $A_\gamma$ and $B_\gamma$ are also uniformly bounded on $L_2$-bounded sets, take $C_\infty \in \bbR$ such that $A_\gamma(w), B_\gamma(w) <C_\infty$ for all $w: \|w\|_{L_2}< C_w$. Then, combining the above estimates and using the fact that $\bar\rho = e^{-\gamma W*\bar\rho}/Z(\gamma, 0)$ we obtain the following bound for the non-linear operator $G$:
\begin{align*}
\MoveEqLeft \|G(w, \gamma)\|^2_{L_2} = \Big\|\bar \rho - \gamma \bar \rho \cdot  W* w + \gamma \bar\rho^2  \int_{\bbS^{n-1}}  W* w dx  -A_\gamma(w)\Big\|^2_{L_2} \\
&=%
C_\infty^2\left(|f(w)|^2 +e^{-2\gamma W*\bar\rho}\frac14C_\xi^2\|(\gamma W*w)^2\|_{L_2}^2\right) + \bar\rho^2\|\gamma W*w \int_{\bbS^{n-1}}  W* w d\sigma\|_{L_2}^2
\\
&\leq \tilde C_1|f(w)|^2 + \tilde C_2\|(\gamma W*w)^2d\sigma\|_{L_2}^2 = o(\|w\|^2_{L_2}).
\end{align*}
\textbf{Higher derivatives.} We now calculate Frechet derivatives of the non-linear term $G$. For $D_w G$ using \eqref{eq:DwF} we get
\begin{align*}
\MoveEqLeft    D_w G(w, \gamma)[v] = D_w\hat{F}(w, \gamma)[v] -v + \gamma Kv \\
&= \gamma Kv +\frac{1}{Z( \gamma, w)}  e^{-\gamma W* (\bar \rho + w)} \cdot \gamma W* v - \frac{ \gamma e^{-\gamma W* (\bar \rho + w)}}{Z( \gamma, w)^2}\int_{\bbS^{n-1}} e^{-\gamma W* (\bar \rho + w)} \cdot W* v d\sigma.
\end{align*}
Note that 
Define operators $Q_1(w):= Z(\gamma, w)^{-1}$, $Q_2(w) := e^{-\gamma W* (\bar \rho + w)}$ and $Q_3(w): = \int_{\bbS^{n-1}} e^{-\gamma W* (\bar \rho + w)} \cdot W* v$ for $v\in L_2(\bbS^{n-1})$, we show that all of them are continuous. To begin with, by Cauchy-Schwartz inequality we get the following uniform bound
\[
|W * w|_\infty \leq \|W\|_{L_2} \|w\|_{L_2}.
\]
Using it, for $Q_2$ we obtain for sufficiently small $\|w_1\|_{L_2}, \|w_1\|_{L_2} < c_w$ the bound
\begin{align*}
    \|Q_2(w_1) - Q_2(w_2)\|^2_{L_2} &= \left\|e^{-\gamma W* (\bar \rho + w_2)} - e^{-\gamma W* (\bar \rho + w_1)}\right\|^2_{L_2} \\
    &= \int e^{-2\gamma W* (\bar \rho + w_2)}\left(1 - e^{-\gamma W*(w_1 - w_2)}\right)^2 \\
    &\leq 2\gamma \sigma(\bbS^{n-1}) \|W\|^2_{L_2}e^{2\gamma\|W\|_{L_2}(1 + c_w)}\|w_1 - w_2\|_{L_2}^2\\
    &= O(\|w_1 - w_2\|_{L_2}^2).
\end{align*}
Similarly for $Q_1$ we then get the estimate
\begin{align*}
    |Q_1(w_1) - Q_1(w_2)| &= \frac{\left|\int e^{-\gamma W* (\bar \rho + w_2)} - e^{-\gamma W* (\bar \rho + w_1)}\right|}{\int e^{-\gamma W* (\bar \rho + w_1)} \cdot \int e^{-\gamma W* (\bar \rho + w_2)}} = \frac{\int e^{-\gamma W* (\bar \rho + w_2)}\left(1 - e^{-\gamma W*(w_1 - w_2)}\right)}{\int e^{-\gamma W* (\bar \rho + w_2)} \cdot \int e^{-\gamma W* (\bar \rho + w_1)}}\\
    &\leq C e^{3\|W\|_{L_2}(1+c_w)}\int\left|1 - e^{-\gamma W*(w_1 - w_2)}\right| \\
    &\leq 2\gamma C \sigma(\calM)\|W\|_{L_2} e^{3\|W\|_{L_2}(1+c_w)} \|w_1-w_2\|_{L_2} = O(\|w_1 - w_2\|_{L_2}).
\end{align*}
Finally, for $Q_3$ assuming that $\|v\|_{L_2}<1$ we obtain
\begin{align*}
    |Q_3(w_1) - Q_3(w_2)| &= \int (Q_1(w_1) - Q_1(w_2)) \cdot W*v \leq \|W\|_{L_2}\int|Q_1(w_1) - Q_1(w_2)| \\
    &\leq 2\gamma \sigma(\bbS^{n-1}) \|W\|_{L_2}e^{\gamma\|W\|_{L_2}(1 + c_w)}\|w_1 - w_2\|_{L_2} = O(\|w_1 - w_2\|_{L_2}). 
\end{align*}
The derivative $D_w G(w, \gamma)[v]$ then takes the form
\[
D_w G(w, \gamma)[v] =\gamma Kv +\gamma Q_1(w)Q_2(w)\cdot W*v - \gamma Q_1(w)Q_2(w)^2Q_3(w),
\]
and, as directly follows from the above estimates, is continuous in $w$ for sufficiently small perturbation $w$ and any finite $\gamma$. Continuity in $\gamma$ follows in similar manner. Calculating $D_\gamma G(w, \gamma)$ and $D_{w,\gamma} G(w, \gamma)[u]$ we obtain the following expressions.
\begin{align*}
D_\gamma G(w,\gamma) &= D_\gamma\left(\frac{1}{Z(\gamma, w)} e^{-\gamma W* (\bar \rho + w)} \right)+  \bar \rho \cdot  W* w -  \bar\rho^2  \int_{\bbS^{n-1}}  W* w d\sigma \\
&=e^{-\gamma W* (\bar \rho + w)} D_\gamma\left(\frac{1}{Z(\gamma, w)}\right) + \frac{1}{Z(\gamma, w)} D_\gamma\left(e^{-\gamma W* (\bar \rho + w)}\right) \\
&\qquad +  \bar \rho \cdot  W* w -  \bar\rho^2  \int_{\bbS^{n-1}}  W* w d\sigma \\
&=e^{-\gamma W* (\bar \rho + w)} \frac{e^{-\gamma W* (\bar \rho + w)} }{Z(\gamma, w)^2} \int e^{-\gamma W*(\bar\rho +w) }\cdot W * (\bar\rho + w) -\frac{\gamma e^{-\gamma W* (\bar \rho + w)}}{Z(\gamma,w)} \\
&\qquad +  \bar \rho \cdot  W* w -  \bar\rho^2  \int_{\bbS^{n-1}}  W* w d\sigma.
\end{align*}
And analogously for the mixed derivative:
\begin{align*}
    D_{w,\gamma} G(w,\gamma)[v] &= Kv +\frac{1}{Z( \gamma, w)}  e^{-\gamma W* (\bar \rho + w)} \cdot W* v  + \gamma W* v\cdot D_\gamma\left(\frac{1}{Z( \gamma, w)}  e^{-\gamma W* (\bar \rho + w)}\right) \\
    &\qquad- \frac{e^{-\gamma W* (\bar \rho + w)}}{Z( \gamma, w)^2}\int_{\bbS^{n-1}} e^{-\gamma W* (\bar \rho + w)} \cdot W* v d\sigma \\
    &\qquad+\gamma D_\gamma\left(\frac{e^{-\gamma W* (\bar \rho + w)}}{Z( \gamma, w)^2}\int_{\bbS^{n-1}} e^{-\gamma W* (\bar \rho + w)} \cdot W* v d\sigma\right).
\end{align*}
A more explicit expression for $D_{w,\gamma} G(w,\gamma)[v]$ can be obtained using the product rule. Continuity of both $D_\gamma G(w,\gamma)$ and $D_{w,\gamma} G(w,\gamma)[v]$ follows from the arguments analogous to the estimates for $D_w G(w, \gamma)$. 
\end{proof}

\begin{proof}[Proof of Lemma \ref{lemma:compactness}]
    Consider a bounded set $U \subset L_2(\calM)$, we aim to show that the images of $U$ under both $A$ and $B$ are relatively compact by applying Arzelà–Ascoli theorem. We start by showing the pointwise boundedness. Using the argument from Lemma \ref{lemma:gibbs-H1} and Hölder's inequality we obtain
    \begin{align*}
        \|Au\|_\infty &= \|W*u\|_\infty \leq \|W\|_\infty \|u\|_{L_1} \leq \|W\|_\infty \|u\|_{L_1} \leq \|W\|_\infty \sqrt{\sigma(\calM)}\|u\|_{L_2}, \\
        \|Bu\|_\infty &= \left\|\sigma \int_\calM W * ud\sigma\right\|_\infty \leq \left\| \int_\calM \|W\|_\infty\sqrt{\sigma(\calM)}\|u\|_{L_2} d\sigma\right\|_\infty = \|W\|_\infty \sigma(\calM)^{3/2}\|u\|_{L_2}.
    \end{align*}
    Note that $Bu$ is a constant function for every $u\in L_2(\calM)$, so the set $B(U)$ is equicontinuous. For $A(U)$ we estimate for $x_1, x_2 \in \calM$
    \begin{align*}
         |(Au)(x_1) - (Au)(x_2)|^2 &= \left|\int_\calM W(x_1, y)u(y)d\sigma(y) - \int_\calM W(x_2, y)u(y)d\sigma(y) \right|^2 \\
        &= \left|\int_\calM \left(W(x_1, y) - W(x_2, y)\right)u(y)d\sigma(y)  \right|^2 \\
        &\leq \int_\calM  \left|W(x_1, y) - W(x_2, y)\right|^2 \|u(y)\|^2d\sigma(y)  \\
    &\leq \|u\|_{L_2}^2\sup_{y\in \calM}\left|W(x_1, y) - W(x_2, y)\right|^2.
    \end{align*}
    Since $W$ is continuous on $\calM\times \calM$ and $\calM$ is compact, $W$ is uniformly continuous. Thus for every $\varepsilon >0$ there exists $\delta>0$ such that for all $x_1, x_2: \ \dist(x_1, x_2) <\delta$ it holds that $\left|W(x_1, y) - W(x_2, y)\right| < \varepsilon$, which guarantees that 
    \[
     |(Au)(x_1) - (Au)(x_2)|^2 \leq \varepsilon^2\|u\|_{L_2}^2,
    \]
    so $A(U)$ is equicontinuous. Thus by Arzelà–Ascoli theorem both $A(U)$ and $B(U)$ are relatively compact sets in the topology of uniform convergence. Since $\calM$ is compact, $A(U)$ and $B(U)$ are also relatively compact in $L_2(\calM)$ implying that both operators are compact.
\end{proof}
\bibliographystyle{myalpha}
\bibliography{biblio}
\end{document}